\documentclass[12pt]{amsart}
\usepackage{latexsym}
\usepackage{amsthm}
\usepackage{amsmath}
\usepackage[dvips]{graphics}
\usepackage{graphicx}
\usepackage{amssymb}
\usepackage{color}
\usepackage{ulem}
\usepackage{epsfig}
\usepackage[all]{xy}
\usepackage{booktabs}

\numberwithin{equation}{section}

\theoremstyle{plain}
\newtheorem{theorem}{Theorem}[section]
\newtheorem{lemma}[theorem]{Lemma}

\newcommand{\rational}{\ensuremath{\mathbb{Q}}} 
\newcommand{\field}{\ensuremath{\mathbb{F}}}    
\newcommand{\real}{\ensuremath{\mathbb{R}}}     
\newcommand{\cplex}{\ensuremath{\mathbb{C}}}    
\newcommand{\integer}{\ensuremath{\mathbb{Z}}}  
\newcommand{\nature}{\ensuremath{\mathbb{N}}}  
\newcommand{\wout}{\setminus}

\newcommand{\exc}{\mathcal{E}}
\newcommand{\ssum}[2] {\overset{#2}{\underset{#1}{\sum}}}  

\begin{document}

\newtheorem*{theorem*}{Theorem}
\newtheorem*{lemma*}{Lemma}
\newtheorem*{claim*}{Claim}
\newtheorem*{exercise*}{Exercise}
\newtheorem*{note*}{Note}
\newtheorem*{example*}{Example}
\newtheorem*{problem*}{Problem}
\newtheorem*{solution*}{Solution}
\newtheorem*{remark*}{Remark}

\newtheorem{corollary}[theorem]{Corollary}
\newtheorem{example}[theorem]{Example}
\newtheorem{conjecture}[theorem]{Conjecture}
\newtheorem{definition}[theorem]{Definition}
\newtheorem{proposition}[theorem]{Proposition}
\newtheorem{remark}[theorem]{Remark}
\newcommand{\Z}{\integer}
\newcommand{\zpinfty}{\integer (p^\infty)}
\newcommand{\x}{\text{\boldmath{$X$}}}
\newcommand{\h}{\mathcal{H}}
\newcommand{\HH}{\text{\boldmath{$H$}}}
\newcommand{\g}{\text{\boldmath{$g$}}}
\newcommand{\G}{\text{\boldmath{$G$}}}
\newcommand{\F}{\text{\boldmath{$F$}}}
\newcommand{\w}{\text{\boldmath{$w$}}}
\newcommand{\re}{\text{Re}}
\newcommand{\vv}{\text{\boldmath{$v$}}}
\newcommand{\order}{\mathcal{O}}
\newcommand{\End}{\text{End}}
\newcommand{\iso}{\cong}                        

\title{ CM liftings of Supersingular Elliptic Curves}
\address{Department of Mathematics\\ Radboud Universiteit Nijmegen\\
Heijendaalseweg 135, 6525 AJ Nijmegen, Netherlands
}
\author{Ben Kane}
\email{bkane@science.ru.nl}
\thanks{This research was conducted while the author was a student at the University of Wisconsin-Madison.}
\date{\today}
\begin{abstract}
Assuming GRH, we present an algorithm which inputs a prime $p$ and outputs the set of fundamental discriminants $D<0$ such that the reduction map modulo a prime above $p$ from elliptic curves with CM by $\order_{D}$ to supersingular elliptic curves in characteristic $p$.  In the algorithm we first determine an explicit constant $D_p$ so that $|D|> D_p$ implies that the map is necessarily surjective and then we compute explicitly the cases $|D|<D_p$.  

Supposant vraie la conjecture de Riemann g\'{e}n\'{e}ralis\'{e}e nous pr\'{e}sentons un algorithme qui, donn\'{e} un nombre pr\'{e}mier p, calcule l'ensemble des discriminants fondamentaux $D<0$ tels que l'application de reduction modulo un premier aux dessus $p$ des courbes elliptiques avec multiplication complexe par $\order_{D}$ vers les courbes elliptiques supersingular en characteristique $p$ est surjective.  Dans l'algorithme, nous d'abord determinons une borne $D_p$ explicite, tel que $|D|> D_p$ implique que l'application est necessairement surjective et puis nous calculons explicitement les cas $|D|<D_p$.

\end{abstract}
\maketitle

\section{Introduction}\label{introsection}
For $D<0$ a fundamental discriminant, consider the imaginary quadratic field $K:=\rational(\sqrt{D})$ with ring of integers $\order_{D}$ and Hilbert class field $H_K$.  From the work of Deuring \cite{Deuring1}, given a prime $p$ which does not split in $\order_{D}$ (i.e. $(D/p)\neq 1$) and an elliptic curve $E/H_K$ with CM by $\order_{D}$ (i.e. $\End_{\overline{K}}(E)\iso \order_D$) the reduction to characteristic $p$ gives a supersingular elliptic curve defined over $\field_{p^2}$.  Using an equidistribution result of Duke and Schulze-Pillot \cite{DukePillot1}, based upon bounds for coefficients of half-integral weight cusp forms by Iwaniec \cite{Iwaniec1} and Duke \cite{Duke1}, combined with an (ineffective) lower bound for the class number $h(D)$ of $\order_{D}$ due to Siegel \cite{Siegel1}, Elkies, Ono and Yang \cite{ElkiesOnoYang1} deduce that the reduction map is surjective for $|D|$ sufficiently large.   Denote the (finite) set of such $D$ for which the reduction map is not surjective by $\exc_p$ and define $\exc_p':= \{|D|: D\in \exc_p\}$.  We will say that $D_p\in \nature$ is a \begin{it}good bound for $p$\end{it} if $\max \exc_p' < D_p$ (suppressing $p$ when the context is clear).

Although $\exc_p$ is finite, no explicit good bound is given above due to the ineffective nature of Siegel's lower bound.  In this paper we will present an algorithm which will input a prime $p$ and return the set $\exc_p$.  This algorithm is conditional upon the Generalized Riemann Hypothesis for Dirichlet L-functions and also the Generalized Riemann Hypothesis for the L-series of weight 2 primitive newforms (henceforth simply denoted GRH).  In particular, the algorithm will terminate unconditionally, but the correctness is dependent on GRH.  The assumption of GRH will allow us to use techniques developed by Ono and Soundarajan \cite{OnoSound1} to explicitly compute a good bound for $p$.
\begin{theorem}\label{BoundTheorem}
Let a prime $p$ be given.  Conditional upon GRH, there is an effectively computable good bound for $p$.
\end{theorem}
Explicitly computing the bound given by Theorem \ref{BoundTheorem} for $p\leq 107$ we obtain the following.
\begin{theorem}\label{Dps}
Assuming GRH, $3.257\times 10^{25}$ is a good bound for $p\leq 107$.  Moreover, Table \ref{GoodBoundTable} in appendix \ref{datasection} contains good bounds for each $p\leq 107$.
\end{theorem}
After obtaining the good bound $D_p$ from Theorem \ref{BoundTheorem}, it only remains to explicitly compute the set of $|D|\leq D_p$ for which the mapping is not surjective.  For each supersingular elliptic curve $E/\field_{p^2}$ we will construct a positive definite (ternary) quadratic form $Q_E$ such that $Q_E$ represents $|D|$ if and only if there exists $E'$ with CM by $\order_{D}$ which reduces to $Q_E$.  Since there are only finitely many supersingular elliptic curves we then merely need to check which $|D|\leq D_p$ are represented by each $Q_E$.  

One may then use the algorithm by Fincke and Pohst \cite{FinckePohst1} to determine a vector of length $|D|$.  The usual implementation returns all vectors so running this algorithm for each $|D|\leq D_p$ to determine $\exc_p$ is $\Omega\left(D_p^{3/2-\epsilon}\right)$ and the calculation quickly becomes infeasible for moderately large $D_p$.  We hence want to take advantage of the fact that we do not need all representations of $|D|$ but rather only one.  In the case where $E$ is defined over $\field_p$ we are able to use a classification result of Ibukiyama \cite{Ibukiyama1} to develop a specialized algorithm which determines more efficiently the set of $|D|<D_p$ which are represented (see Section \ref{smallDsection}).  This algorithm has allowed us to compute the full set $\exc_p$ for $p=11,17,$ and $19$. 
\begin{theorem}\label{explicitthm}
Assuming GRH, the following hold.
\begin{enumerate}
\item \label{peq11} The set $\exc_{11}$ is given by 
\begin{eqnarray*}
\exc_{11}'&=& \left\{  3, 4, 11, 67,88, 91, 163, 187, 232, 235, 427 499, 595,627,\right.\\
&&\left.  715, 907,1387, 1411, 3003, 3355, 4411, 5107,6787, 10483, 11803 \right\}.
\end{eqnarray*}
\item \label{peq17}The set $\exc_{17}$ satisfies $\#\exc_{17}=91$ and $\max \exc_{17}' = 89563$.
\item \label{peq19}The set $\exc_{19}$ satisfies $\#\exc_{19}=45$ and $\max \exc_{19}' = 27955$.
\end{enumerate}
\end{theorem}

Having established such surjectivity results, one may ask whether similar results can be shown about the multiplicity of the reduction map.  This question was addressed and an unconditional but ineffective solution was given by Elkies, Ono, and Yang \cite{ElkiesOnoYang1}.  

Define the Hilbert class polynomial $\h_D(x)\in\integer[x]$ as the unique monic polynomial whose roots are precisely the $j$-invariants of the elliptic curves with complex multiplication by $\order_{D}$.  These roots are referred to as \begin{it}singular moduli of discriminant\end{it} $D$.  The degree of the Hilbert class polynomial is $h(D)$.  Define further $S_p(x)\in \field_p[x]$ to be the polynomial with roots precisely the $j$-invariants of the supersingular elliptic curves of characteristic $p$.
\begin{theorem*}[Elkies-Ono-Yang \cite{ElkiesOnoYang1}]
For a prime $p$ and $t\in\nature$, every sufficiently large fundamental discriminant $D<0$ for which $p$ does not split in $\order_{D}$ satisfies
$$
S_p(x)^t\mid \h_D(x)
$$
over $\field_p[x]$.  
\end{theorem*}
Here the implied constant depends on $p$ and $t$.  Their result states that for sufficiently large $D$ there are at least $t$ nonisomorphic elliptic curves with CM by $\order_D$ which reduce to each supersingular elliptic curve of characteristic $p$.  We are again able to obtain an effective but conditional result of this nature.  For a supersingular elliptic curve $E$, define $w_E$ to be the number of automorphisms of $E$ and take the canonical measure
$$
\mu(E) :=\frac{1/w_E}{\sum_{E'}1/w_{E'}},
$$
where the sum is taken over all supersingular elliptic curves of characteristic $p$.  We will denote the minimal value of this measure by $\mu_p$.
\begin{theorem}\label{multthm}
Assume GRH.  For a prime $p$ and $0<c<1$ there is a effectively computable constant $D_{p,c}\in \nature$ such that every fundamental discriminant $D<0$ with $|D|\geq D_{p,c}$ for which $p$ is not split satisfies 
$$
S_p(x)^{c\mu_p h(D)}\mid \h_D(x)
$$
over $\field_p[x]$.  
\end{theorem}
Since $h(D)\to \infty$ effectively as $D\to -\infty$ (Oesterl\'{e} \cite{Oesterle1} unconditionally showed the growth is $\Omega(\log(|D|)^{1-\epsilon})$, but Siegel \cite{Siegel1} obtained $\Omega(|D|^{1/2-\epsilon})$ conditional on GRH), we also get an effective but conditional version of Elkies, Ono, and Yang's result by choosing for each $t\in\nature$ an integer $D_{p,t}\geq D_{p,c}$ large enough so that $c\mu_p h(D)>t$ for every $|D|\geq D_{p,t}$. 
In Section \ref{QuadPaperSection} we will see that Theorem \ref{multthm} reduces to the same argument given to show Theorem \ref{BoundTheorem} but we will not explicitly compute the bound here.

The paper will begin by reviewing the connection between theta series and $\exc_p$ in Section \ref{CMliftsection}.  In Section \ref{QuadPaperSection}, we review how the bound for coefficients of theta series is obtained.  Given the connection from Section \ref{CMliftsection}, this gives a good bound for $p$, dependent on numerically calculating certain constants.  In Section \ref{calcsection}, we fix a basis and decompose a certain space of modular forms in order to calculate some of the constants obtained from Section \ref{QuadPaperSection}.  Furthermore, we give explicit algorithms for calculating the remaining constants.  In Section \ref{smallDsection}, we use a trick based on the Ibukiyama's classification \cite{Ibukiyama1} of the set of supersingular elliptic curves defined over $\field_p$, in order to calculate the set of $|D|<D_p$ which are generated by $Q_{E}$.  Finally, in Appendix \ref{datasection} we give the data obtained by explicitly implementing the algorithms from Sections \ref{calcsection} and \ref{smallDsection} for $p\leq 107$.

\begin{it}Acknowledgements.\end{it}
The author would like to thank T.H. Yang for his help and guidance and would also like to thank K. Bringmann, K. Ono, and J. Rouse for useful comments.

\section{CM Liftings of Supersingular Elliptic Curves and Theta Series}\label{CMliftsection}

For a supersingular elliptic curve $E$ we will say that $D_E$ is a \begin{it}good bound\end{it} for $E$ if $E$ is in the image of the reduction map for every $|D|>D_E$.  Hence we will piecewise determine a good bound $D_p$ for $p$ by determining a good bound $D_E$ for each supersingular elliptic curve $E/\field_{p^2}$ and then taking $D_p:=\max_E \ D_E$, relying on the fact that there are only finitely many supersingular elliptic curves (up to isomorphism).
This also aids in computing the elements $D\in \exc_p$ with $|D|<D_p$, since we only need to check all $|D|<D_E$ for each curve, and not up to the larger bound $D_p$.  The theory involved in determining $D_E$ goes through quaternion algebras, quadratic forms, theta series, and modular forms.  For background information on elliptic curves a good reference is Silverman's book \cite{Silverman1}.  A good reference for quaternional algebras is Vigernas's book \cite{Vigneras1}, while Ono's book \cite{Ono1} contains a good introduction to modular forms.  Good sources of information about quadratic forms can be found in Jones' book \cite{Jones1} and O'Meara's book \cite{o'meara}.

Let $E/\field_{p^2}$ be a supersingular elliptic curve.  An elliptic curve $\widetilde{E}$ with CM by $\order_D$ is a \begin{it}CM (by $\order_{D}$) lift\end{it} of $E$ if the reduction of $\widetilde{E}$ equals $E$.  We will now review the connection between CM liftings and theta series.  Let $R_E:=End(E)$ be a maximal order of the quaternion algebra $B_p/\rational$ ramified precisely at $p$ and infinity.  For $p$ inert in $\order_{D}$ (resp.\ $p$ ramified in $\order_{D}$) there is a one-to-one (resp.\ two-to-one) correspondence between lifts of $E$ and embeddings of $\order_{D}$ in $R_E$.  Gross and Zagier \cite[Prop. 2.7]{GrossZagier1} cover the case of $p$ inert and Elkies \cite[p. 168]{Elkies1} extends this to the case where $p$ is ramified.  Let $L_E:=\{ x\in \integer+2R_E| tr(x)=0\}$ be the \begin{it}Gross lattice\end{it} with the associated positive definite ternary quadratic form $Q_E$ given by the reduced norm on $L_E$.

For a positive definite ternary quadratic form $Q$ we will define the \begin{it}theta series\end{it} of $Q$ by
$$
\theta(z):=\theta_Q(z) := \ssum{a,b,c\in\integer}{}q^{Q(a,b,c)},
$$
where $z$ is in the upper half plane and $q:=e^{2\pi i z}$.  Denote the theta series of $Q_E$ by 
$$
\theta_{E}(z) := \sum_{x\in L_E} q^{Q_E(x)} = \sum_{\substack{d<0\\ d\equiv 0,1\pmod 4}} a_{E}(d)q^{|d|}.
$$
Noting that $w_E=\# R_E^*$, Gross \cite[Prop. 12.9, p. 172]{Gross1} has shown that $a_E(D)$ equals $\frac{w_E}{\# \order_{D}^*}$ times the number of embeddings of $\order_{D}$ into $R_E$.  It is hence sufficient to proceed by bounding the coefficients of the theta series from below, and showing that they must be positive whenever $|D|>D_E$.  However, Gross showed that $\theta_E$ is a weight $3/2$ modular form in Kohnen's plus space (see \cite{Kohnen1} or \cite[p. 54]{Ono1} for a definition) of level $4p$ and explicit bounds (conditional on GRH) for coefficients of theta series in this space were established by the author \cite{Kane2}.  The methods used to obtain these bounds will be reviewed in Section \ref{QuadPaperSection}.

\section{Background}\label{QuadPaperSection}

This section is a brief summary of the theory used to bound the coefficients of $\theta_E$.  The theta series is first decomposed into a linear combination of an Eisenstein series and a basis of weight $3/2$ Hecke eigenforms.  Using an isomorphism to weight 2 cusp forms, the coefficients of the Hecke eigenforms are then compared with the central values of quadratic twists of $L$-series of weight 2 newforms.  The central value of these $L$-series are bounded in the authors's generalization \cite{Kane2} of Ono and Soundararajan's paper \cite{OnoSound1} in terms of constants which we will introduce here.  Section \ref{calcsection} will be devoted to explicitly determining the basis of weight $3/2$ Hecke eigenforms, the isomorphism to weight 2 newforms, and explicitly bounding these constants.

For $k\in\integer$ and $N\in\nature$ we will denote the space of (holomorphic) modular forms of weight $k$ and level $N$ by $M_k(N)$, the cuspidal subspace by $S_k(N)$, and the space of newforms by $S_k^{\text{new}}(N)$.  Moreover, for $k\in \frac{1}{2}\integer\wout \integer$ we will denote Kohnen's plus space of level $4N$ by $M_k^+(4N)$ and the cuspidal subspace by $S_{k}^+(4N)$.  For a modular form $g$, we denote the $n$-th Fourier coefficient by $a_g(n)$.  

Let $g_1,\dots g_r$ be a basis of Hecke eigenforms in $S_{3/2}^+(4p)$ and define the Eisenstein series
$$
H_{\theta}(z):= \frac{12}{p-1} + \sum_{d<0,\ d\equiv 0,1\pmod{4}} \frac{12}{p-1}\cdot \frac{1-\left(\frac{d}{p}\right)}{2} H(d/p^{2e_d}) q^{|d|},
$$
where $H(d)$ is the Hurwitz class number and $p^{e_d}$ is the highest power of $p$ dividing the square part of $|d|$.  In particular, for $D<0$ a fundamental discriminant, $H(D)$ equals the class number of the imaginary quadratic field $\rational(\sqrt{D})$ divided by the number of units modulo $\pm 1$.  Gross\cite{Gross1} has shown that $\theta_E-H_{\theta}\in S_{3/2}^+(4p)$.  We hence decompose our theta series as
$$
\theta=H_{\theta}+\ssum{i=1}{t_p-1} b_ig_i,
$$
for some $b_i\in \cplex$.  Here $t_p$ is the number of distinct conjugacy classes of maximal orders of $B_p$, called the \begin{it}type number\end{it}.

For a fundamental discriminant $D<0$ with corresponding Kronecker character $\chi_{D}$, define the $|D|$-th Shimura correspondence $S_{|D|}$  by 
$$
\ssum{n=1}{\infty}\frac{a_{g|S_{|D|}}(n)}{n^s}:=L(\chi_D,s)\ssum{n=1}{\infty}\frac{a_g(|D|n^2)}{n^s}
$$
for every $g\in S_{3/2}^+(4p)$.  Here and throughout we denote the image of $g$ under an operator $T$ by $g|T$.  Shimura \cite{Shimura1} showed that $g|S_{|D|}\in S_2(4p)$ and that $S_{|D|}$ commutes with every Hecke operator, namely 
$$
f|_{3/2} T_{\ell^2} |S_{|D|} = (f|S_t) |_2 T_{\ell}
$$
for every $f\in S_{3/2}^+(4p)$ and every prime $\ell\neq p$.  Let $t_i\in\integer_{>0}$ be minimal with $4\mid t_i$ and $-t_i$ a fundamental discriminant satisfying $a_{g_i}(t_i)\neq 0$.  If $r_{j}$ are chosen so that $\sum_{j=1}^{t_p-1} r_{j} a_{g_i}(t_j)\neq 0$ for every $1\leq i\leq t_p-1$, then Kohnen \cite{Kohnen1} has shown that the linear combination of Shimura correspondences
\begin{equation}\label{ShimLiftEqn}
S:=\sum_{i=1}^{t_p-1} r_{i} S_{t_i},
\end{equation}
called a \begin{it}Shimura lift\end{it}, forms an isomorphism from $S_{3/2}^+(4p)$ to $S_2^{\text{new}}(p) = S_2(p)$ which sends Hecke eigenforms to Hecke eigenforms. 

Denote the Shimura lift $g_i|S$ by $G_i$.  For a fundamental discriminant $D<0$ and $\re(s)>1$, we denote the $L$-series of $\chi:=\chi_D$ by 
$$
L(s):=L(\chi,s) := \sum_{n=1}^{\infty} \frac{\chi(n)}{n^s}
$$
and for $\re(s)>\frac{3}{2}$ we denote the $L$-series of $G_i$ twisted by $\chi$ as
\begin{equation}
L_i(s):=L(G_i,D,s):=\ssum{n=1}{\infty}\frac{\chi(n)a_{G_i}(n)}{n^s}.
\end{equation}
The conductor of $L_i(s)$ is $q:=p|D|^2$.  Denote by $m_i$ the smallest integer such that $a_{g_i}(m_i)\neq 0$ with $(p,m_i)=1$.  

Using the fact that $G_i$ is a Hecke eigenform, the author \cite{Kane2} showed that any fundamental discriminant $D<0$ satisfying $a_{\theta}(|D|)=0$ must also satisfy
\begin{equation}\label{setupeqn}
\frac{12}{(p-1)\pi 2^{\frac{v_p(|D|)}{2}}}\cdot |D|^{\frac{1}{4}} \leq \sqrt{\ssum{i=1}{t_p-1}|b_i|^2}\sqrt{ \ssum{i=1}{t_p-1}c_i\frac{L_i(1)}{L(1)^2}},
\end{equation}
where 
\begin{equation}\label{cieqn}
c_i:=\frac{|a_{g_i}(m_i)|^2}{L(G_i,m_i,1)m_i^{\frac{1}{2}}}
\end{equation}
is a constant which comes from taking the ratio of the $|D|$-th coefficient and the $m_i$-th coefficient in the Kohnen-Zagier formula \cite{KohnenZagier1}.    Define 
\begin{equation}
F(s):=F_i(s):=\left(\frac{\sqrt{q}}{2\pi}\right)^{s-1}\frac{L_i(s)\Gamma(s)}{L(s)L(2-s)}
\end{equation}
and choose $1<\sigma<\frac{3}{2}$.  To obtain a contradiction from equation (\ref{setupeqn}) for $|D|$ sufficiently large, it remains to bound $F(1)=\frac{L_i(1)}{L(1)^2}$ from above.  

Since equation (\ref{setupeqn}) is obtained by assuming $a_{\theta}(|D|)\leq 0$ and rearranging, we can similarly assume $a_{\theta}(|D|)\leq c a_{H_{\theta}}(|D|)$ for a constant $0<c<1$ and obtain equation (\ref{setupeqn}) with the left hand side multiplied by $1-c$.  Hence Theorem \ref{multthm} is also reduced to bounding $F(1)$ and the details are left to the reader.  

We will describe briefly how Ono and Soundararajan \cite{OnoSound1} bounded $F(1)$.  By the functional equation of $F(s)$ and the Phragm\'{e}n-Lindel\"{o}f principle we have
$$
F(1)\leq \max_{t\in\real} F(\sigma+it),
$$
so it suffices to bound $F(\sigma+it)$ from above for every $t\in\real$.

For a real number $\x>0$ and an $L$-series $\widetilde{L}(s)$ with $c>0$ real chosen such that $s+c$ is in the region of absolute convergence, consider the integral
\begin{equation}\label{Linteqn}
\int_{c-i\infty}^{c+i\infty} \frac{\widetilde{L}'}{\widetilde{L}}(s+w) \Gamma(w) \x^w dw.
\end{equation}
On the one hand, if $\widetilde{L}(s)=\sum_{n=1}^{\infty} \frac{a(n)}{n^s}$ in the region of convergence, then (\ref{Linteqn}) can be computed as the sum 
\begin{equation}\label{eval1eqn}
\sum_{n=1}^{\infty} \frac{\Lambda(n) a(n)}{n^s} e^{-n/\x}
\end{equation}
using the fact that 
$$
\int_{c-i\infty}^{c+i\infty} \Gamma(w) \left(\frac{\x}{n}\right)^w dw = \sum_{m=0}^{\infty}\frac{(-1)^m}{m!}\left(\frac{n}{\x}\right)^m=e^{-n/\x},
$$
which follows by shifting the integral $\re(w)\to -\infty$ and counting the residues at negative integers.  On the other hand, we can shift the original integral to the left and count the contribution from residues at each of the poles.  The contribution from $w=0$ gives $\frac{\widetilde{L}'}{\widetilde{L}}(s)$.  The assumption of GRH allows us to determine the real part of all of the poles coming from $\frac{\widetilde{L}'}{\widetilde{L}}$, since these correspond to zeros of $\widetilde{L}(s)$.  Rearranging the resulting equation gives a formula for $\frac{\widetilde{L}'}{\widetilde{L}}(s)$ which we integrate to get a formula for $\log|\widetilde{L}(s)|$, as shown by Ono and Soundararajan \cite[Lemmas 1-2]{OnoSound1}.

Using the above argument with $\widetilde{L}(s)=L(s)$, we define for $\re(s)>1$ the integral of equation (\ref{eval1eqn}) by
\begin{equation}\label{defineGeqn}
\G(s,\x):=\ssum{n=1}{\infty}\frac{\Lambda(n)\chi(n)}{n^s\log(n)}e^{-n/\x}.
\end{equation}
Similarly, with $\lambda_i$ defined so that $\frac{L_i'}{L_i}(s)=\ssum{n=1}{\infty}\frac{\lambda_i(n)\chi(n)}{n^s}$ when $\re(s)>3/2$, define 
\begin{equation}\label{F1eqn}
\F_1(s,\x):=\ssum{n=1}{\infty}\frac{\lambda_i(n)\chi(n)}{n^s} e^{-n/ \x}
\end{equation}
and $$
\F(w,\x):=\ssum{n=1}{\infty} \frac{\lambda_i(n)\chi(n)}{n^w\log(n)}e^{-n/\x} = \int \F_1(w,\x) dw.
$$

For $s=\sigma+it$, $s_0=2-\sigma+it$ and $s_2=\sigma_2+it$ for any choice $\sigma<\sigma_2<2$, the following bound \cite{Kane2} was obtained for $F(s)$.  For certain explicit constants $c_{\theta,\sigma,\x,1}$, $c_{\theta,\sigma,\x,t,1}$, $c_{\theta,\sigma,\x,m,1}$, $c_{\theta,\sigma,\x,2}$, $c_{\theta,\sigma,\x,t,2}$, and $c_{\theta,\sigma,\x,q,2}$
\begin{multline*}
\log|F(s)|\leq \frac{\x}{\x+1}\F(s,\x) - \frac{\x((2 + \gamma(\x))\alpha(\x)-\beta(\x)}{(\x+1)(1+\gamma(\x))}\F_1(s_2,\x)\\
 -  \frac{\x}{\x-1-\delta(\x)\x}\left(\ \re(\G(s_0,\x))-\re(\G(s,\x))\right)
+ c_{\theta,\sigma,\x,2} + c_{\theta,\sigma,\x,t,2} + c_{\theta,\sigma,\x,q,2}\\
 - \left( c_{\theta,\sigma,\x,1} + c_{\theta,\sigma,\x,t,1} + c_{\theta,\sigma,\x,m,1}\right) + \log |\Gamma(s)| - 2\log|L(s)|,
\end{multline*}
where $\alpha(\x)$, $\beta(\x)$, $\gamma(\x)$ and $\delta(\x)$ are defined in the author's previous paper \cite{Kane2} in equations (7.2), (7.1), the line directly proceeding (7.1), and the first equation in Section 6, respectively.  Moreover, in Section 8 of that paper, explicit bounds in terms of $\Gamma$-factors are given for $\alpha(\x)$, $\beta(\x)$, $\gamma(\x)$, and $\delta(\x)$.  

The decay in $\Gamma(s)$ cancels polynomial growth in $t$ from $c_{\theta,\sigma,\x,t,i}$.  Since $\sigma>1$, $L(s)$ converges absolutely, so we can explicitly calculate a bound for $2\log|L(s)|$ as well.  The other terms involving $\F$, $\G$, and $\F_1$ are also dealt with \cite{Kane2} (Here, we use cancellation in the sums for small $n$ between terms from $2\log|L(s)|$, and then bound the remaining terms separately).  We will give further details in Section \ref{calcsection} of how to compute better bounds for these constants.  

Therefore, the main goal of this paper will be to decompose Kohnen's plus space, make a choice of $g_i$, determine a Shimura lift, and then calculate $b_i$ and $c_i$.  This is described in Section \ref{calcsection}.  Moreover, in feasible cases we must determine an algorithm to determine whether $|D|$ is represented by a fixed form $Q$.  A specialized algorithm is given in Section \ref{smallDsection} to determine this whenever the corresponding elliptic curve is defined over $\field_p$.

\section{Algorithm to compute $D_E$ and $D_p$}\label{calcsection}

This section is broken into four main subsections.  We first determine the set of theta series $\theta_E$ for every supersingular elliptic curve $E/\field_{p^2}$.  We then determine a basis of Hecke eigenforms $\{ g_i: i \in \{1,\dots,t_p-1\} \}$ for the subspace of $S_{3/2}^{+}(4p)$ generated by    these theta series and express the cuspidal part of the theta series as a linear combination of these eigenforms.  The third step will be to compute an explicit Shimura lift $S$ from $S_{3/2}^+(4p)$ to $S_2(p)$ and finally we compute the constants corresponding to $G_i=g_i|S$.

\subsection{Calculating the Theta Series.} 
Let $p$ be a prime and $C>0$ be an integer.  We will describe here how to obtain the quadratic forms $Q_E$ and the first $C$ coefficients of $\theta_E$ for every supersingular elliptic curve $E/\field_{p^2}$.  If $C$ is chosen too small for the remaining calculations, then we will simply double $C$ and rerun the calculations.

We begin by calculating the maximal orders $R_E$ using Kohel's \cite{Kohel1} algorithm built into MAGMA.  These are obtained by first using a function to calculate a single maximal order $R$, then calculating all left ideal classes $I_i$ of $R$, and finally calculating the right order $R_i$ of $I_i$, which will give a full set of maximal orders.  We will represent $R_E$ as a $4\times 4$ matrix over $\rational$.  Let the standard basis of $B_p$ over $\rational$ be given by $1, \alpha,\beta$ and $\gamma=\alpha\beta=-\beta\alpha$ with $\alpha^2=p$ and $\beta^2=q$ for some prime $q\equiv 3\pmod {8}$ for which $\left(\frac{-q}{p}\right)=-1$ (cf. \cite[p. 1]{Ibukiyama1}).  Then a matrix $A$ will correspond to the $\integer$-module which is generated by $A_{1,j} + A_{2,j}\alpha +A_{3,j}\beta + A_{4,j}\gamma$.  It is then straightforward to compute $L_E$ and the method of Fincke and Pohst \cite{FinckePohst1} may be used to compute the coefficients $a_E(d)$ for every $d<C$.

\subsection{Decomposition of $\theta_E$.}
We will now compute a basis of (cuspidal) Hecke eigenforms $g_1,\dots g_{t_p-1}$ of the space spanned by all of our theta series and then decompose the cuspidal part $g_E:=\theta_E-H_{\theta}$ in terms of these Hecke eigenforms.  This will give us the coefficients $b_i$ from section \ref{QuadPaperSection}.  

A computational solution to the decomposition problem for integral weight forms follows from the work of Stein \cite{Stein1} on modular symbols.  We recall that Kohnen \cite{Kohnen1} has shown that the Hecke algebra on $S_{3/2}^{+}(4p)$ is isomorphic to the Hecke algebra on $S_2^{\text{new}}(p)=S_2(p)$.  Furthermore, Sturm has shown for $S_2(p)$ that a finite set of Hecke operators generates the Hecke algebra and has given an effectively computable bound $N$ so that $\left\{ T_{n}| n\leq N\right\}$ generates the Hecke algebra \cite{Sturm1}.  The Hecke eigenspaces of distinct normalized Hecke eigenforms on $S_2(p)$ are at most one dimensional (that is, $S_2(p)$ satisfies \begin{it}(strong) multiplicity one\end{it}, cf. \cite[p. 29]{Ono1}).  Therefore $S_{3/2}^+(4p)$ also satisfies multiplicity one.

We first note that the space generated by our theta series is invariant under the action of the Hecke algebra, so that the space is generated by the set of $g_E$ (which have coefficients in $\rational$).  Since $S_{3/2}^+(4p)$ has multiplicity one, we can diagonalize the Hecke operators $T:=T_{n^2}$ simultaneously.  We only need to diagonalize the operators for $n\leq N$, where $N$ is the Sturm bound on $S_2(p)$ as above.  Checking computationally, it appears as though a single $g_E=:g$ always generates the entire space under the action of the Hecke algebra (we have checked for all $p<1000$) and we will demonstrate how to obtain a basis of Hecke Eigenforms in this case.  This assumption is not really restrictive because one merely needs to follow the same argument for the forms $g_{E_1},g_{E_2},\dots$ until the dimension equals $t_p-1$.  We may also choose a particular $T$ such that $g|T^m$ generates the entire subspace (see Stein's book \cite[p. 167]{Stein2}).  

Given $g$ and $T$ as above we calculate $g|T^m$ for every $0\leq m<t_p$.  Then using linear algebra over $\rational$ we obtain $g|T^{t_p-1}$ as a linear combination of $g|T^m$ with $0\leq m<t_p-1$, giving a matrix $M_T$ with rational coefficients.  Let $F$ be the Galois splitting field over $\rational$ of the characteristic polynomial of $M_T$.  Since $g|T^m$ has coefficients in $\rational$ we can diagonalize $M_T$ to obtain Hecke eigenforms with coefficients in $F$.  Since $S_{3/2}^+(4p)$ has multiplicity one and $T$ generates the whole space, the eigenspace of a given eigenvalue has dimension one.  Hence we may calculate with linear algebra over $F$ the unique eigenform $g_i$ with eigenvalue $\lambda_i$.

We can now decompose each $g_E$ as a linear combination of the $g_i$ by linear algebra over $F$.  This gives the desired coefficients $b_i\in F$ in the decomposition
$$
g_E = \sum_{i=1}^{t_p-1} b_i g_i.
$$

\subsection{Finding a Shimura Lift.}
Having established the Hecke eigenforms $g_i$, we will now choose $t_i$ and $r_i$ as in equation (\ref{ShimLiftEqn}) to establish a Shimura lift.  We will recursively choose $t_{\ell}$ and $r_{\ell}$ such that 
$$
S(\ell):= \sum_{j=1}^{\ell} r_j S_{t_j}
$$
satisfies $g_i|S(\ell)=0$ if and only if $a_{g_i}(t_j)=0$ for every $1\leq j\leq \ell$.  

At each step we choose $i$ smallest such that $g_i|S(\ell-1)=0$.  We then choose $t_{\ell}$ to be the smallest integer with $4\mid t_{\ell}$, $-t_{\ell}$ is a fundamental discriminant, and $a_{g_i}(t_{\ell})\neq 0$, noting existence has been shown by Kohnen \cite{Kohnen1}.  It remains to choose $r_{\ell}$ such that for every $k$ we have $\sum_{j=1}^{\ell} r_j a_{g_k}(t_j)=0$ if and only if $a_{g_k}(t_j)=0$ for $1\leq j\leq \ell$.  Since $F=\rational(\alpha)$ is a number field we may consider $F$ as a vector space over $\rational$ with basis $\alpha^i$.  

Let $k$ be given such that $a_{g_k}(t_j)\neq 0$ for some $j$.  If $a_{g_k}(t_{\ell})=0$, then we know by inductive hypothesis that $g_k|S(\ell-1)+r_{\ell}S_{t_{\ell}}\neq 0$ for any $r_{\ell}$.  If $a_{g_k}(t_{\ell})\neq 0$, then writing it in terms of the basis, we have $a_{g_k}(t_{\ell})=\sum d_m \alpha^m$ with some $d_m\neq 0$.  Computing 
$$
\sum_{j=1}^{\ell-1} r_{j} a_{g_k}(t_j)
$$
and rewriting in terms of the basis, we write the coefficient $e_m\in \rational$ of $\alpha^m$.  We then take 
$$
r_{\ell,k}:=\left|\frac{e_m}{d_m}\right| + \frac{1}{2}.
$$
Taking $r_{\ell}:=\max_k r_{\ell,k}$, we have 
$$
|r_{\ell}d_m| > |e_m|
$$
and hence $r_{\ell}d_m+e_m\neq 0$.  It follows that $g_k|S(\ell)\neq 0$ because the coefficient of $\alpha^m$ in the first Fourier coefficient is nonzero.  Since $k$ was arbitrary, we have $g_i|S(\ell)=0$ if and only $a_{g_i}(t_j)=0$ for every $1\leq j\leq \ell$, as desired.  We then terminate if $g_i|S(\ell)\neq 0$ for every $i$ and otherwise continue the recursion.

\subsubsection{Calculating $c_i$}
Recall first that 
$$
c_i=\frac{|a_{g_i}(m_i)|^2}{L(G_i,m_i,1)m_i^{1/2}}
$$
for $m_i$ a fixed integer such that $a_{g_i}(m_i)\neq 0$ and 
$m_i\neq 0 \pmod{p}$ and $G_i=g_i|S$.  We may simply choose $m_i$ to be the smallest such integer.  In the bounds that we obtain it will suffice to bound $|c_i|$ from above.

We have already shown how to calculate $a_{g_i}(m_i)$, so it remains to calculate $L(G_i,m_i,1)$.  We use the following formula of Cremona \cite{Cremona1},
$$
L(G_i,m_i,1)=\ssum{n=1}{\infty} 2a_L(n)\chi(n)e^{-2\pi \frac{n}{m_i\sqrt{p}}}.
$$
Calculating the partial sum up to a fixed bound $N$ and noting by Deligne's optimal bound \cite{Deligne1} that $|a_L(n)|\le \sigma_0(n)n^{\frac{1}{2}}$, we may bound the error easily by pulling the absolute value inside the sum for $n>N$.

\subsection{Calculating the other constants}

These constants are actually fairly easy to calculate once we show clearly 
where they come from, given the theoretical results stated in the author's previous paper \cite{Kane2}.  The methods involved and notation used are similar to those used by Ono and Soundararajan \cite{OnoSound1}.  

Most of the constants obtained are explicit in terms of $\Gamma$ and $\zeta$ factors along the real line, but we need some work to calculate the terms involving $\F$, $\F_1$, and $\G$ (coming from equation (\ref{eval1eqn}).  Define $v(n,\x)$ by
\begin{multline*}
v(n,\x):=c_{\theta,\x,1,\F}\frac{\lambda_i(n)e^{-n/\x}}{n^{\sigma}} + c_{\theta,\x,1,\F_1}\frac{\log(n)\lambda_i(n)e^{-n/\x}}{n^{\sigma_2}}\\
 -  c_{\theta,\x,2,G} \left( \frac{\Lambda(n)e^{-n/\x}}{n^{\sigma_0}} - \frac{\Lambda(n)e^{-n/\x}}{n^{\sigma}}\right),
\end{multline*}
where $\sigma=\re(s)$, $\sigma_0=\re(2-s)$, and $\sigma_2=\re(s_2)$, so that 
\begin{multline*}
\ssum{n=2}{\infty} \re\left( \frac{\chi(n)}{n^{it}\log(n)} v(n,\x)\right) = c_{\theta,\x,1,\F}\re(\F(s,\x)) + c_{\theta,\x,1,\F_1} \re(\F_1(s_2,\x))\\ 
- c_{\theta,\x,2,G} \re(\G(s_0,\x)-\G(s,\x)).
\end{multline*}

We will bound the following to get a constant independent of the variables 
involved.  From above, we need to bound 
\begin{equation}\label{boundLminus}
 -2\log |L(s)| + 2\ssum{n=2}{N_0} \re\left(\frac{\chi(n)\Lambda(n)}{n^s\log(n)}\right).
\end{equation}
Noting that $s$ is in the region of absolute convergence,
$$
\log(|L(s)|)=\ssum{n=2}{\infty} \re\left(\frac{\chi(n)\Lambda(n)}{n^s}\log(n)\right).
$$
Then equation (\ref{boundLminus}) becomes
$$
-2\log(|L(s)|) + 2\ssum{n=2}{N_0} \re\left(\frac{\chi(n)\Lambda(n)}{n^s\log(n)}\right)= - 2\ssum{n=N_0+1}{\infty} \re\left(\frac{\chi(n)\Lambda(n)}{n^{s}\log(n)}\right).
$$
Therefore, taking the absolute value inside the sum gives
$$
2\left|\ssum{n=N_0+1}{\infty} \frac{\chi(n)\Lambda(n)}{n^{s}\log(n)}\right|\leq 2\ssum{n=N_0+1}{\infty} \frac{\Lambda(n)}{n^{\sigma}\log(n)}=  2 \log(|\zeta(\sigma)|) -\ssum{n=2}{N_0+1}\frac{\Lambda(n)}{n^{\sigma}\log(n)},
$$
and this final finite sum and $\zeta(\sigma)$ are easily computed.

We also need a bound for the constants depending on $t$, the imaginary part of $s$.  We use the functional equation of the $\Gamma$ factor to remove the growth from these terms.  Since the growth is logarithmic in $t$ we easily obtain
\begin{equation}\label{boundtpart}
\log |\Gamma(s)| + c_{\theta,\x,1,t} - c_{\theta,\x,2,t} \leq \log |\Gamma(\sigma+r)|
\end{equation}
for some $r\in\nature$.

A computer is then used to bound
\begin{equation}\label{compbound}
\ssum{n=2}{N_0} \re\left( \frac{\chi(n)}{n^{it}\log(n)}\left( v(n,\x) -\frac{2\Lambda(n)}{n^{\sigma}} \right)\right).
\end{equation}
Notice that the term we are subtracting is exactly the term being added in equation (\ref{boundLminus}).  The only nonzero terms are $p$ powers, so the maximum is taken by calculating $ \frac{1}{\log(n)} \left( v(n,\x) -\frac{2\Lambda(n)}{n^{\sigma}}\right)$ for each $n=p^k$ and then noting that either $\chi(p^k)=\chi(p)^k$, which is either one or alternates.  Finding the $t$ which maximizes this sum for each $p$, independent of whether the sum alternates or not, gives the bound, since we then add up the absolute value of each of these terms together.

It remains to bound the terms coming from equation (\ref{eval1eqn}) with $n$ large.  We will hence look at 
\begin{equation}\label{boundFGpart}
\ssum{n=N_0+1}{\infty} \re\left( \frac{\chi(n)}{n^{it}\log(n)}( v(n,\x)) \right).
\end{equation}
Notice first, since $\sigma_2>\sigma$, that for $n$ sufficiently (namely we choose $N_0$ such that this occurs for $n>N_0$) the term from the $\F_1$ part of $v(n,\x)$ satisfies the bound
$$
c_{\theta,\x,1,\F_1} \frac{\log(n)}{n^{\sigma_2}}\leq \frac{c_{\theta,\x,1,\F}}{n^{\sigma}}.
$$
Therefore, we see that 
$$
|v(n,\x)|\leq e^{-n/\x}\left( 2c_{\theta,\x,1,\F}\frac{|\lambda_i(n)|}{n^{\sigma}} + c_{\theta,\x,2,G}\Lambda(n) \left(\frac{1}{n^{\sigma_0}}-\frac{1}{n^{\sigma}}\right)\right).
$$
Since $\lambda_i(n)\leq 2\sqrt{n}\log(n)$, we can further bound this by 
$$
c_{\theta,\x,v} \frac{\Lambda(n)}{n^{min(\sigma-1/2,\sigma_0)}}e^{-n/x}.
$$
In \cite{Kane2}, we have shown for $\alpha=\text{min}(\sigma-1/2,\sigma_0)$ an explicit constant $c_{N_0}$ such that
\begin{multline}
\HH(\alpha,\x):=\ssum{n=N_0+1}{\infty} \frac{\Lambda(n)}{n^\alpha\log(n)}e^{-n/x}\\
 \leq \frac{e^{-N_0/\x}}{N_0^\alpha\log(N_0)}(c_{N_0}N_0-\psi(N_0)) + \frac{c_{N_0}\x^{1-\alpha} }{\log(N_0)} \Gamma(1-\alpha, N_0/\x).
\end{multline}
We then calculate the incomplete Gamma factor $\Gamma(1-\alpha,N_0/\x)$ (cf. \cite{incomplete}), giving the desired bound.

\section{Determining CM Lifts for $|D|<D_E$ when $E$ is Defined over $\field_p$} \label{smallDsection}

In this section, we give an algorithm to determine whether $E/\field_p$ is in the image of the reduction map from elliptic curves with CM by $\order_{D}$ for a fixed $D$ to deal with $|D|<D_E$.  It is based on a classification of $R_E$ given by Ibukiyama \cite{Ibukiyama1}.

\subsection{Calculating which $|D|$ are Represented by the Gross Lattice}

\begin{lemma}\label{Fplemma}
Let $E$ be a supersingular elliptic curve defined over $\field_p$, $L_E$ be its associated Gross lattice, and $R_E^{0}$ be the lattice of trace zero coefficients.  Then there exists a lattice $L$ satisfying $L_E\subseteq L \subset R_E^{0}$ such that the reduced norm on $L$ is 
$$
Q(x,y,z)=px^2 + (by^2 + fyz + cz^2 ).
$$
 \end{lemma}

\begin{proof}
Since $E$ is defined over $\field_p$, Ibukiyama \cite{Ibukiyama1} has shown that $R_E$ is of one of the following two types,
\begin{equation}\label{maxorder1}
R(q,r):=\integer + \integer\frac{1+\beta}{2} + \integer\frac{\alpha(1+\beta)}{2} +\integer \frac{(r+\alpha)\beta}{q}
\end{equation}
or
\begin{equation}\label{maxorder2}
R'(q,r'):=\integer + \integer\frac{1+\alpha}{2} + \integer\beta +\integer \frac{(r'+\alpha)\beta}{2q},
\end{equation}
where $q$ is a prime satisfying $q\equiv 3\pmod{8}$ and $\left(\frac{-q}{p}\right)=-1$, $\alpha^2=-p$, $\beta^2=-q$, $\alpha\beta=-\beta\alpha$, $r^2+p\equiv 0\pmod{q}$ and $r'^2+p\equiv 0\pmod{4q}$ in the case when $p\equiv 3\pmod{4}$. 

For $R_E=R(q,r)$ with basis $\frac{1+\beta}{2}$, $\gamma_1:=\beta$, $\gamma_2:=\frac{\alpha(1+\beta)}{2}$, and $\gamma_3:=\frac{(r+\alpha)\beta}{q}$, $R_E^0$ is generated by $\gamma_1,\gamma_2,\gamma_3$ while $L_E$ is generated by $\gamma_1,2\gamma_2,2\gamma_3$.  We take $L$ to be the lattice generated by $\gamma_1,2\gamma_2,\gamma_3$.  If an arbitrary element of $L$ is written $x\gamma_1+2y\gamma_2+z\gamma_3$, then the change of variables $x':=x-ry$, $y':=z+qy$ and $z':=y$ gives the reduced norm
$$
 p(x')^2 + \frac{r^2+p}{q}(y')^2 +p(z')^2 +2rx'y',
$$
as desired.  Changing $z$ to $2z$ above implies that $z'\equiv y' \pmod{2}$, so that the reduced norm on $L_E$ is precisely the quadratic form given above with $z'\equiv y'\pmod{2}$.

If $R_E= R'(q,r')$, we have a simpler task.  In this case, the reduced norm on $R_E^{0}$ is simply
$$
px^2 + q y^2 + \frac{(r')^2+p}{4q}z^2 + r'yz.
$$
To get the elements of the Gross lattice, we simply multiply $y$ and $z$ by $2$ to get
$$
Q'(x,y,z):=px^2 + (4q) y^2 + \frac{(r')^2+p}{q}z^2 + (4r')yz.
$$
\end{proof}

Given Lemma \ref{Fplemma}, the reduced norm on $L_E$ is either of the form
$$
Q(x',y',z'):=q(x')^2 + \frac{r^2+p}{q}(y')^2 + p(z')^2 +2rx'y',
$$
with $z'\equiv y'\pmod{2}$, or
$$
Q'(x,y,z):=px^2 + (4q) y^2 + \frac{(r')^2+p}{q}z^2 + (4r')yz.
$$

To check if an integer $n$ is represented, we first set two integers $M$ and 
$N$ and do a precomputation for efficiency.  For $Q$, we do a precomputation of the two sets 
$$
SE_{M}:=\{ n\leq M: n=q(x')^2 + \frac{r^2+p}{q}(y')^2+ 2rx'y', y' \text{ even}\},
$$
and analogously
$$
SO_{M}:=\{ n\leq M: n=q(x')^2 + \frac{r^2+p}{q}(y')^2+ 2rx'y', y' \text{ odd}\}.
$$
Since we know that, with $x'$ fixed, the minimum value is obtained at $xdiv:=(q-\frac{rq}{p+r^2})x'^2$, we run $x'$ from $0$ to $\left(\frac{M}{xdiv}\right)^{1/2}$ and then $y'$ from 0 to $\frac{2rx' + \sqrt{4r^2x' - 4\frac{p+r^2}{q}\cdot (q(x')^2-M)}}{2\frac{p+r^2}{q}}$, and simply calculate $n=Q(x',y',0)$.  If $y'$ is odd, we add $n$ to $SO_M$, and if $y'$ is even then we add $n$ to $SE_M$.

Similarly, for $Q'$, we calculate 
$$
S_{M}:=\{ n\leq M: n=Q'(0,y,z)\}.
$$

Given $SE_{M}$ and $SO_{M}$, we now calculate
$$
T_{N,M}:=\{ n\leq N : n=m+p(z')^2, m\in SE_{M}\text{ and $z'$ even, or } m\in SO_{M}\text{ and $z'$ odd}\}.
$$
Notice that, if we define
$$
T_{N}:=\{n\leq N: n=Q(x',y',z'), y'\equiv z'\pmod{2} \},
$$
then $T_{M}\subseteq T_{N,M}\subseteq T_{N}$.  Therefore, for every $n\in T_{N,M}$, we know $n\in T_{N}$, and for every $n\notin T_{N,M}$ with $n\leq M$, we know $n\notin T_{N}$.  Since we expect that after a low bound $M$ we will not have any such eligible elements which are not in $T_N$, we can set $M$ lower for optimization purposes.

We now describe the algorithm to calculate $T_{N,M}$.  For each eligible $|D|\leq N$, we check from $z'=\left(\frac{|D|-M}{p}\right)^{1/2}$ to $z'=\left(\frac{|D|}{p}\right)^{1/2}$.  For each $z'$, if $z'$ is even, then we check if $|D|-p(z')^2\in SE_M$, and if $z'$ is odd, we check if $|D|-p(z')^2\in SO_{M}$.  If so, then we add $|D|$ to $T_{N,M}$.  The algorithm for $Q'$ is entirely analogous, only needing to check membership in $S_{M}$ instead of breaking it up into the even and odd cases.  We know that $np^2\in T_{N,M}$ if and only if $n\in T_{N,M}$, so we can skip checking these cases.

We shall show that the running time for this function is $O(p+NM^{1/2})$.  We need time $O(M)$ to calculate $SE_M$ and $SO_M$.  Calculating the modulus of $p$ which are eligible takes time $O(p)$.  For each $D$, we have to check at most $M^{1/2}$ possible $z'$.  Therefore, since there are $O(N)$ such $D$, this calculation takes $O(NM^{1/2})$.  Thus, the overall running time is $O(M+p+NM^{1/2})=O(p+NM^{1/2})$ (since we will choose $N>p$, we have $O(NM^{1/2})$).  

Notice that for an individual $n\notin T_{N,M}$, we can check membership in $T_N$ in $O(N^{1/2})$ time by calculating checking membership in $SE_{N}$ and $SO_N$ (or $S_N$ for $O'$).  By doing this as a precomputation again, we get a running time of $O(N + N^{1/2}E)$ where $E$ is the number of exceptional $D\notin T_{N,M}$.  Therefore, if we choose $M$ so that $E<(NM)^{1/2}$, then we can calculate $T_N$ in $O(NM^{1/2})$.

\appendix
\section{Data}\label{datasection}
We will now use the algorithm from Section \ref{calcsection}  to compute good bounds for $p\leq 107$, using $\x=455$, $\sigma=1.15$, $N_0=1000$, and $\sigma_2=1.3256$ (These were chosen by a binary search for $\sigma$ and a heuristically based search for $\sigma_2$ given $\sigma$.).  Tables \ref{goodtable1}, \ref{goodtable2}, \ref{goodtable3}, and \ref{goodtable4} will give the good bounds for $E$.  Combining the good bounds for every $E/\field_{p^2}$ we obtain good bounds for $p$ in Table \ref{GoodBoundTable}.  For each maximal order $R_E$, we will list the prime $p$, then the size of the field $\field_q$ ($q=p$ or $q=p^2$) which the corresponding elliptic curve is defined over.  We will then list the corresponding ternary quadratic form as $[a,b,c,d,e,f]=ax^2+by^2+cz^2+dxy+exz+fyz$.  We then list a good bound $D_0$ for $E$ which suffices when $(D,p)=1$, and a good bound $D_1$ which also suffices when $p\mid D$.  We separate these cases since a better bound is obtained for $D$ relatively prime to $p$ and skipping $(D,p)=1$ is a computational gain.  We omit here the primes 3, 5, 7, and 13, since we have $D_p=1$ trivially.  

\begin{table}
\caption{Good Bounds $D_E$ for $E/\field_{p^2}$.}
\begin{tabular}{lccll} 
\toprule
$p$ & $\#\field_q$  &   \textnormal{Quadratic Form} & $D_0$ & $D_1$ \\
\midrule
\midrule
11&  $p$  &    \textnormal{[4,11,12,0,4,0]}   &    $1.813\times 10^8$      &             $3.163\times 10^8$\\
11 & $p$  &    \textnormal{[3,15,15,-2,2,14]} &    $5.142\times 10^8$   &               $8.973\times 10^9$\\
\midrule
17 & $p$    &  \textnormal{[7,11,20,-6,4,8]} &      $1.002\times 10^{10}$    &              $1.748\times 10^{11}$\\
17 & $p$ &      \textnormal{[3,23,23,-2,2,22]} &     $8.652\times 10^{13}$    &        $1.510\times 10^{15}$\\
\midrule
19 & $p$ &      \textnormal{[7,11,23,-2,6,10]} &     $3.020\times 10^{9}$          &       $5.270\times 10^{10}$\\
19 & $p$ &      \textnormal{[4,19,20,0,4,0]} &       $9.198\times 10^{11}$        &      $1.606\times 10^{13}$\\
\midrule
23 & $p$ &      \textnormal{[8,12,23,4,0,0]} &       $7.459\times 10^{10}$         &     $3.700\times 10^{11}$\\
23 & $p$ &      \textnormal{[4,23,24,0,4,0]} &       $2.050\times 10^{14}$      &    $2.522\times 10^{15}$\\
23 & $p$ &      \textnormal{[3,31,31,-2,2,30]} &     $8.297\times 10^{14}$      &   $6.955\times 10^{15}$\\
\midrule
29 & $p$ &      \textnormal{[11,12,32,8,4,12]} &     $6.739\times 10^{11}$         &     $1.008\times 10^{12}$\\
29 & $p$ &      \textnormal{[8,15,31,4,8,2]} &       $3.836\times 10^{13}$       &     $3.130\times 10^{14}$\\
29 & $p$ &      \textnormal{[3,39,39,-2,2,38]} &     $1.900\times 10^{16}$     &     $1.550\times 10^{17}$\\
\midrule
31 & $p$ &      \textnormal{[7,19,36,-6,4,16]} &     $3.836\times 10^{12}$       &    $4.359\times 10^{13}$\\
31 & $p$ &      \textnormal{[8,16,31,4,0,0]} &       $1.245\times 10^{13}$      &     $2.069\times 10^{14}$\\
31 & $p$ &      \textnormal{[4,31,32,0,4,0]} &       $8.558\times 10^{14}$    &     $1.008\times 10^{16}$\\
\midrule
37 & ${p^2}$ &      \textnormal{[15,20,23,-4,14,8]} &    $2.101\times 10^{11}$       &        $3.667\times 10^{12}$\\
37 & $p$ &      \textnormal{[8,19,39,4,8,2]} &       $6.399\times 10^{13}$       &    $1.117\times 10^{15}$\\
\midrule
41 & $p$ &      \textnormal{[11,15,47,-2,10,14]} &   $1.834\times 10^{14}$       &    $3.201\times 10^{15}$\\
41 & $p$ &      \textnormal{[12,15,44,8,12,4]} &     $2.520\times 10^{14}$       &    $7.830\times 10^{15}$\\
41 & $p$ &      \textnormal{[7,24,47,4,2,24]} &      $1.967\times 10^{15}$     &    $1.447\times 10^{16}$\\
41 & $p$ &      \textnormal{[3,55,55,-2,2,54]} &     $4.375\times 10^{17}$   &    $3.579\times 10^{18}$\\
\bottomrule
\end{tabular}
\label{goodtable1}
\end{table}

\begin{table}
\caption{Good Bounds $D_E$ for $E/\field_{p^2}$.}
\begin{tabular}{lccll} 
\toprule
$p$ & $\#\field_q$  &   \textnormal{Quadratic Form} & $D_0$ & $D_1$ \\
\midrule
\midrule
43 & ${p^2}$ &      \textnormal{[15,23,24,2,8,12]} &     $2.565\times 10^{12}$       &       $4.476\times 10^{13}$\\
43 & $p$ &      \textnormal{[11,16,47,4,2,16]} &     $4.056\times 10^{13}$        &   $7.079\times 10^{14}$\\
43 & $p$ &      \textnormal{[4,43,44,0,4,0]} &       $1.364\times 10^{16}$       &   $2.379\times 10^{17}$\\
\midrule
47 & $p$ &      \textnormal{[12,16,47,4,0,0]} &      $1.492\times 10^{14}$       &   $2.604\times 10^{15}$\\
47 & $p$ &      \textnormal{[7,27,55,-2,6,26]} &       $1.080\times 10^{15}$      &   $1.527\times 10^{16}$\\
47 & $p$ &      \textnormal{[8,24,47,4,0,0]}&     $1.056\times 10^{15}$     &   $1.842\times 10^{16}$\\
47 & $p$ &      \textnormal{[4,47,48,0,4,0]} &       $2.339\times 10^{17}$    &   $4.082\times 10^{18}$\\
47 & $p$ &      \textnormal{[3,63,63,-2,2,62]} &     $3.702\times 10^{17}$   &   $6.461\times 10^{18}$\\
\midrule
53 & ${p^2}$ &      \textnormal{[20,23,32,-12,4,20]} &    $1.174\times 10^{14}$   &       $1.428\times 10^{15}$\\
53 & $p$ &      \textnormal{[12,19,56,8,12,4]} &     $4.015\times 10^{15}$     &   $6.101\times 10^{16}$\\
53 & $p$ &      \textnormal{[8,27,55,4,8,2]} &       $5.825\times 10^{16}$    &   $4.883\times 10^{17}$\\
53 & $p$ &      \textnormal{[3,71,71,-2,2,70]} &     $6.918\times 10^{18}$  &   $5.467\times 10^{19}$\\
\midrule
59 & $p$ &      [15,16,63,4,2,16] &     $6.695\times 10^{13}$       &   $7.662\times 10^{14}$\\
59 & $p$ &      [15,19,64,-14,8,12] &   $6.695\times 10^{13}$       &   $7.662\times 10^{14}$\\
59 & $p$ &      [7,35,68,-6,4,32] &     $4.612\times 10^{14}$      &   $2.426\times 10^{15}$\\
59 & $p$ &      [12,20,59,4,0,0] &      $2.811\times 10^{15}$     &   $4.492\times 10^{16}$\\
59 & $p$ &      [4,59,60,0,4,0] &       $1.106\times 10^{17}$    &   $1.174\times 10^{18}$\\
59 & $p$ &      [3,79,79,-2,2,78] &     $7.295\times 10^{17}$   &   $1.166\times 10^{19}$\\
\midrule
61 & ${p^2}$ &      \textnormal{[23,24,32,16,4,12]} &    $8.254\times 10^{14}$   &      $6.927\times 10^{15}$\\
61 & $p$ &      \textnormal{[7,35,71,-2,6,34]} &     $5.007\times 10^{15}$       &  $2.545\times 10^{16}$\\
61 & $p$ &      \textnormal{[8,31,63,4,8,2]} &       $5.892\times 10^{16}$      &  $2.803\times 10^{17}$\\
61 & $p$ &      \textnormal{[11,23,68,-6,8,20]} &    $5.240\times 10^{16}$     &  $3.797\times 10^{17}$\\
\midrule
67 & ${p^2}$ &      \textnormal{[23,24,35,8,2,12]} &     $5.517\times 10^{14}$   &      $9.628\times 10^{15}$\\
67 & ${p^2}$ &      \textnormal{[15,36,39,-4,14,16]} &   $1.105\times 10^{15}$   &      $1.928\times 10^{16}$\\
67 & $p$ &      \textnormal{[16,19,71,12,16,6]} &    $1.207\times 10^{16}$      &  $1.987\times 10^{17}$\\
67 & $p$ &     \textnormal{[4,67,68,0,4,0]} &       $2.623\times 10^{18}$   &  $2.642\times 10^{19}$\\
\midrule
71 & $p$ &      \textnormal{[15,20,76,8,4,20]} &     $7.313\times 10^{16}$    &  $5.485\times 10^{17}$\\
71 & $p$ &      \textnormal{[12,24,71,4,0,0]} &      $3.235\times 10^{16}$   &  $2.936\times 10^{17}$\\
71 & $p$ &      \textnormal{[15,19,79,-2,14,18]} &   $1.343\times 10^{17}$    &  $5.962\times 10^{17}$\\
71 & $p$ &      \textnormal{[16,20,71,12,0,0]} &     $1.693\times 10^{17}$    &  $1.667\times 10^{18}$\\
71 & $p$ &      \textnormal{[8,36,71,4,0,0]} &       $1.450\times 10^{17}$   &  $2.531\times 10^{18}$\\
71 & $p$ &      \textnormal{[4,71,72,0,4,0]} &       $2.876\times 10^{19}$  &  $1.379\times 10^{20}$\\
71 & $p$ &      \textnormal{[3,95,95,-2,2,94]} &     $2.725\times 10^{19}$  &  $2.191\times 10^{20}$\\
\bottomrule
\end{tabular}
\label{goodtable2}
\end{table}

\begin{table}
\caption{Good Bounds $D_E$ for $E/\field_{p^2}$.}
\begin{tabular}{lccll} 
\toprule
$p$ & $\#\field_q$  &   Quadratic Form & $D_0$ & $D_1$ \\
\midrule
\midrule
73 & ${p^2}$ &      \textnormal{[15,39,40,2,8,20]} &     $8.979\times 10^{15}$    &     $1.147\times 10^{17}$\\
73 & ${p^2}$ &      \textnormal{[20,31,44,-12,4,28]} &   $6.740\times 10^{16}$   &     $4.422\times 10^{17}$\\
73 & $p$ &      \textnormal{[7,43,84,-6,4,40]} &     $1.025\times 10^{17}$      &  $5.334\times 10^{17}$\\
73 & $p$ &      \textnormal{[11,28,80,8,4,28]} &     $1.767\times 10^{18}$     &  $1.452\times 10^{19}$\\
\midrule
79 & ${p^2}$ &  \textnormal{[23,31,44,18,16,20]} &   $2.150\times 10^{15}$  &      $2.481\times 10^{16}$\\
79 & $p$ &      \textnormal{[16,20,79,4,0,0]} &      $2.923\times 10^{16}$     &  $3.458\times 10^{17}$\\
79 & $p$ &      \textnormal{[19,20,84,16,8,20]} &    $5.009\times 10^{16}$    &  $8.741\times 10^{17}$\\
79 & $p$ &      \textnormal{[11,31,87,-10,6,26]} &   $1.112\times 10^{17}$    &  $1.305\times 10^{18}$\\
79 & $p$ &      \textnormal{[8,40,79,4,0,0]} &       $1.169\times 10^{17}$    &  $1.503\times 10^{18}$\\
79 & $p$ &      \textnormal{[4,79,80,0,4,0]} &       $6.499\times 10^{18}$   &  $1.121\times 10^{20}$\\
\midrule
83 & ${p^2}$ &      [23,31,44,-14,8,12] &   $4.054\times 10^{15}$   &     $6.477\times 10^{16}$\\
83 & $p$ &      [12,28,83,4,0,0] &      $1.721\times 10^{16}$      & $2.591\times 10^{17}$\\
83 & $p$ &      [7,48,95,4,2,48] &      $3.913\times 10^{16}$      & $6.251\times 10^{17}$\\
83 & $p$ &      [16,23,87,12,16,6] &    $8.775\times 10^{16}$      & $1.328\times 10^{18}$\\
83 & $p$ &      [11,31,92,-6,8,28] &    $1.574\times 10^{16}$      & $2.514\times 10^{18}$\\
83 & $p$ &      [3,111,111,-2,2,110] &  $4.776\times 10^{18}$    & $7.089\times 10^{19}$\\
83 & $p$ &      [4,83,84,0,4,0] &       $6.461\times 10^{18}$    & $1.033\times 10^{20}$\\
\midrule
89 & ${p^2}$ &  \textnormal{[23,31,48,2,12,16]} &    $3.896\times 10^{17}$  &    $4.145\times 10^{18}$\\
89 & $p$ &      \textnormal{[19,23,95,-18,10,14]} &  $1.236\times 10^{19}$  &   $2.906\times 10^{19}$\\
89 & $p$ &      \textnormal{[15,27,96,-14,8,20]} &   $2.636\times 10^{19}$  &   $5.543\times 10^{19}$\\
89 & $p$ &      \textnormal{[12,31,92,8,12,4]} &     $4.535\times 10^{19}$ & $1.108\times 10^{20}$\\
89 & $p$ &      \textnormal{[15,24,95,4,2,24]} &     $1.052\times 10^{20}$  &   $1.811\times 10^{20}$\\
89 & $p$ &      \textnormal{[7,51,103,-2,6,50]} &    $2.994\times 10^{20}$ &   $3.541\times 10^{20}$\\
89 & $p$ &      \textnormal{[3,119,119,-2,2,118]} &  $1.017\times 10^{22}$&  $6.887\times 10^{22}$\\
\midrule
97 & ${p^2}$ &      \textnormal{[23,39,51,-22,6,14]} &   $1.241\times 10^{17}$  &     $6.289\times 10^{17}$\\
97 & ${p^2}$ &      \textnormal{[15,52,55,-4,14,24]} &   $4.517\times 10^{17}$   &     $2.630\times 10^{18}$\\
97 & $p$ &      \textnormal{[7,56,111,4,2,56]} &     $2.204\times 10^{18}$    &   $4.357\times 10^{18}$\\
97 & ${p^2}$ &      \textnormal{[20,39,59,-4,8,38]} &    $5.923\times 10^{18}$  &     $1.541\times 10^{19}$\\
97 & $p$ &      \textnormal{[19,23,104,-14,12,16]} & $2.188\times 10^{19}$   &   $7.815\times 10^{19}$\\
\midrule
101 & ${p^2}$ &     [32,39,44,-12,28,20] &  $8.477\times 10^{15}$ &       $3.603\times 10^{16}$\\
101 & $p$ &     [12,35,104,8,12,4] &    $1.709\times 10^{17}$   &   $1.223\times 10^{18}$\\
101 & $p$ &     [15,28,108,8,4,28] &    $1.572\times 10^{18}$   &   $3.193\times 10^{18}$\\
101 & $p$ &     [15,27,111,-2,14,26] &  $5.261\times 10^{17}$   &   $3.388\times 10^{18}$\\
101 & $p$ &     [8,51,103,4,8,2] &      $2.948\times 10^{18}$  &   $7.940\times 10^{18}$\\
101 & $p$ &     [7,59,116,-6,4,56] &    $2.341\times 10^{18}$  &   $1.015\times 10^{19}$\\
\bottomrule
\end{tabular}
\label{goodtable3}
\end{table}

\begin{table}
\caption{Good Bounds $D_E$ for $E/\field_{p^2}$.}
\begin{tabular}{lccll}
\toprule
$p$ & $\#\field_q$  &   Quadratic Form & $D_0$ & $D_1$ \\
\midrule
\midrule
101 & $p$ &     [11,39,111,-10,6,34] &  $4.559\times 10^{18}$  &   $2.415\times 10^{19}$\\
101 & $p$ &     [3,135,135,-2,2,134] &  $9.667\times 10^{19}$ &   $5.296\times 10^{20}$\\
\midrule
103 & ${p^2}$ &     [23,36,59,-4,22,16] &   $1.076\times 10^{16}$  &      $1.620\times 10^{16}$\\
103 & $p$ &     [16,28,103,12,0,0] &    $9.459\times 10^{15}$      &  $4.236\times 10^{16}$\\
103 & ${p^2}$ &     [15,55,56,2,8,28] &     $4.016\times 10^{16}$ &      $5.313\times 10^{16}$\\
103 & $p$ &     [19,23,111,-10,14,18] & $1.645\times 10^{17}$    &  $5.558\times 10^{17}$\\
103 & $p$ &     [7,59,119,-2,6,58] &    $1.765\times 10^{17}$    &  $1.861\times 10^{18}$\\
103 & $p$ &     [8,52,103,4,0,0] &      $1.032\times 10^{18}$    &  $2.160\times 10^{18}$\\
103 & $p$ &     [4,103,104,0,4,0] &     $2.647\times 10^{19}$  &  $8.748\times 10^{19}$\\
\midrule
107 & ${p^2}$ &   [35,39,44,-18,32,4] &     $1.769\times 10^{16}$  &  $9.442\times 10^{16}$\\
107 & ${p^2}$ &   [23,40,56,16,4,20] &     $1.352\times 10^{16}$  &  $2.102\times 10^{17}$\\
107 & ${p}$ &   [16,27,111,-4,16,2] &     $7.861\times 10^{16}$  &  $1.256\times 10^{18}$\\
107 & ${p}$ &   [12,36,107,4,0,0] &     $1.061\times 10^{17}$  &  $1.694\times 10^{18}$\\
107 & ${p}$ &   [19,23,116,-6,16,20] &     $9.625\times 10^{17}$  &  $5.827\times 10^{18}$\\
107 & ${p}$ &   [11,39,119,-2,10,38] &     $1.105\times 10^{18}$  &  $1.732\times 10^{19}$\\
107 & ${p}$ &   [4,107,108,0,4,0] &     $4.853\times 10^{19}$  &  $4.368\times 10^{20}$\\
107 & ${p}$ &   [3,143,143,-2,2,142] &     $1.102\times 10^{20}$  &  $1.761\times 10^{21}$\\
\midrule
109 & ${p^2}$ &   \textnormal{[32,44,47,20,28,36]} &     $4.420\times 10^{16}$  &  $7.714\times 10^{17}$\\
109 & ${p^2}$ &   \textnormal{[23,39,59,10,14,22]} &     $5.539\times 10^{16}$  &  $9.666\times 10^{17}$\\
109 & ${p}$ &   \textnormal{[8,55,111,4,8,2]} &     $3.843\times 10^{17}$  &  $5.604\times 10^{18}$\\
109 & ${p^2}$ &   \textnormal{[24,39,56,16,12,4]} &     $8.005\times 10^{18}$  &  $1.397\times 10^{20}$\\
109 & ${p}$ &   \textnormal{[11,40,119,4,2,40]} &     $4.199\times 10^{19}$  &  $7.329\times 10^{20}$\\
109 & ${p}$ &   \textnormal{[19,23,119,-2,18,22]} &     $1.341\times 10^{20}$  &  $1.841\times 10^{21}$\\
\midrule
113 & ${p^2}$ &   \textnormal{[35,39,47,-6,34,10]} &     $1.141\times 10^{18}$  &  $1.133\times 10^{19}$\\
113 & ${p^2}$ &   \textnormal{[23,40,59,8,2,20]} &     $2.158\times 10^{18}$  &  $2.062\times 10^{19}$\\
113 & ${p^2}$ &   \textnormal{[20,47,68,-12,4,44]} &     $4.539\times 10^{18}$  &  $3.297\times 10^{19}$\\
113 & ${p}$ &   \textnormal{[23,24,119,20,10,24]} &     $3.219\times 10^{19}$  &  $1.853\times 10^{20}$\\
113 & ${p}$ &   \textnormal{[19,24,119,4,2,24]} &     $1.649\times 10^{19}$  &  $2.877\times 10^{20}$\\
113 & ${p}$ &   \textnormal{[12,39,116,8,12,4]} &     $1.610\times 10^{19}$  &  $1.029\times 10^{20}$\\
113 & ${p}$ &   \textnormal{[3,151,151,-2,2,150]} &     $1.105\times 10^{22}$  &  $1.041\times 10^{23}$\\
\bottomrule
\end{tabular}
\label{goodtable4}
\end{table}
For $N\in\nature\cup\{\infty\}$, let $\exc_E^N$ be the set of positive integers $n\leq N$ with $p^2\nmid n$ not represented by $Q_E$.  We omit those $n$ with $p^2\mid n$ since $p$ is an anisotropic prime and hence $n$ is represented if and only if $\frac{n}{p^2}$ is represented.  In Tables \ref{exctable1}, \ref{exctable2}, \ref{exctable3}, \ref{exctable4}, \ref{exctable5}, and \ref{exctable6} we will list $\exc_E^N$ computed using the method described in section \ref{smallDsection} when $E$ is defined over $\field_p$ and otherwise using the standard method \cite{FinckePohst1}.  For each elliptic curve we have chosen $N_0$ and $N_1$ and compute $\exc_E^{N_0}$ and $\{ n\in \exc_E^{N_1}|  n \equiv 0\pmod{p}\}$.  When $\#\exc_E^N$ is small we will list the full set $\exc_E^N$, while we will otherwise simply list $\#\exc_E^N$ and $\max (\exc_E^N)$.  Although we are only able to determine $\exc_p$ for $p=11,17,19$ under GRH, we are able to determine $\exc_E:=\exc_E^{\infty}$ for a number of forms, which we will denote by an asterisk next to the form.

\begin{table}
\caption{The set of exceptions $\exc_E^N$.}
\begin{tabular}{llcl}
\toprule
$p$   & Quadratic Form & $N_0/N_1$    & $\exc_E$ or ($\#\exc_E$ and $\max(\exc_E)$)\\
\hline
\midrule
11  & [4,11,12,0,4,0]$^*$       & $3\times 10^9$  & 3, 67, 235, 427\\
\bottomrule
\end{tabular}
\label{exctable1}
\end{table}

\begin{table}
\caption{The set of exceptions $\exc_E^N$.}
\begin{tabular}{llcl}
\toprule
$p$   & Quadratic Form & $N_0/N_1$    & $\exc_E^{N}$ or ($\# \exc_E^{N}$ and $\max(\exc_E^N)$)\\
\midrule
\midrule
11  & [3,15,15,-2,2,14]$^*$     & $10^{10}$  & 4, 11, 88, 91, 163, 187, 232, 499, 595, 627, 715,\\
& & &   907, 1387, 1411, 3003, 3355, 4411, 5107,\\
& & &   6787, 10483, 11803\\
\midrule
17 &  [7,11,20,-6,4,8]$^*$ &    $2\times 10^{11}$ &  3, 187, 643\\
\midrule
17 &  [3,23,23,-2,2,22]$^*$ &    $9\times 10^{13}$ /  &      $\# \exc_E^N=88$ $\max(\exc_E^N)=89563$\\
& &  $2\times 10^{15}$ &\\
\midrule
19 &  [7,11,23,-2,6,10]$^*$ & $10^{11}$ &  4, 19, 163, 760, 1051\\
\midrule
19 &      [4,19,20,0,4,0]$^*$ &    $10^{12}$/  & 7, 11, 24, 43, 115, 123, 139, 228, 232, 267, \\
& &$2\times 10^{13}$ &  403, 424, 435, 499, 520, 568, 627, 643, 691, \\
& & &  883, 1099, 1411, 1659, 1672, 1867, 2139, \\
& & &   2251, 2356, 2851, 3427, 4123, 5131, 5419, \\
& & &    5707,  6619, 7723, 8968, 12331, 22843, 27955\\
\midrule
23 &  [8,12,23,4,0,0]$^*$ &    $4\times 10^{11}$ & 3,4,27, 115, 123,163,403,427, 443, 667,\\
& & &  1467, 2787, 3523\\
\midrule
23 &  [4,23,24,0,4,0] &    $3\times 10^9$ & $\#\exc_E^N =78$, $\max(\exc_E^N)=72427$\\
\midrule
23 &  [3,31,31,-2,2,30] &   $3\times 10^9$ & $\#\exc_E^N=196$, $\max(\exc_E^N)=286603$\\
\midrule
29 & [11,12,32,8,4,12]$^*$ &  $7\times 10^{11}$ /   &     $\#\exc_E^N=24$, $\max(\exc_E^N)=22243$\\
 & & $2\times 10^{12}$ & \\
\midrule
29 & [8,15,31,4,8,2]$^*$ &    $4\times 10^{13}$/  &     $\#\exc_E^N=23$, $\max(\exc_E^N)=7987$\\
& & $4\times 10^{14}$  &\\
\midrule
29 & [3,39,39,-2,2,38] &  $10^{9}$   &     $\#\exc_E^N=382$, $\max(\exc_E^N)=1107307$\\
\midrule
31 &  [8,16,31,4,0,0]$^*$ &  $4\times 10^{12}$/      &     $\#\exc_E^N=36$, $\max(\exc_E^N)=17515$\\
& & $5\times 10^{13}$ & \\
\midrule
31 &  [7,19,36,-6,4,16]$^*$ & $2\times 10^{13}$/       &   $\#\exc_E^N=29$, $\max(\exc_E^N)=15283$\\
& & $3\times 10^{14}$ & \\
\midrule
31 &  [4,31,32,0,4,0]  & $10^{11}$    &  $\#\exc_E^N=166$, $\max(\exc_E^N)=174003$\\
\midrule
37 &  [15,20,23,-4,14,8] & $10^{9}$     &  8,19,43,163,427,723,2923,3907\\ 
\midrule
37 & [8,19,39,4,8,2]$^*$    & $6.5\times 10^{13}$/ & $\#\exc_E^N=55$, $\max(\exc_E^N)=24952$\\
& & $2\times 10^{15}$ & \\
\midrule
41 & [12,15,44,8,12,4] &    $10^{10}$   &    $\#\exc_E^N=60$, $\max(\exc_E^N)=82123$\\
\bottomrule
\end{tabular}
\label{exctable2}
\end{table}

\begin{table}
\caption{The set of exceptions $\exc_E^N$.}
\begin{tabular}{llcl}
\toprule
$p$   & Quadratic Form & $N_0/N_1$    & $\exc_E^{N}$ or ($\# \exc_E^{N}$ and $\max(\exc_E^N)$)\\
\midrule
\midrule
41 & [11,15,47,-2,10,14] &   $10^{10}$  &    $\#\exc_E^N=65$, $\max(\exc_E^N)=48547$\\
\midrule
41 & [7,24,47,4,2,24] &    $3\times 10^{9}$  &   $\#\exc_E^N=82$, $\max(\exc_E^N)=83107$\\
\midrule
41 & [3,55,55,-2,2,54] &    $10^{10}$  &  $\#\exc_E^N=896$, $\max(\exc_E^N)=5017867$\\
\midrule
43 & [15,23,24,2,8,12] &     $3.6\times 10^{10}$       &       4, 11, 16, 52, 67, 187, 379, 403, 568, 883,\\
& & &  1012, 2347, 2451\\
\midrule
43 & [11,16,47,4,2,16]$^*$ & $4.5\times 10^{13}$/ & $\#\exc_E^N=81$, $\max(\exc_E^N)=73315$\\
& & $8\times 10^{14}$ & \\
\midrule
43 &      [4,43,44,0,4,0] &   $10^{9}$  & $\#\exc_E^N=439$, $\max(\exc_E^N)=1079467$\\
\midrule
47 & [12,16,47,4,0,0]  &  $10^{9}$  &   $\#\exc_E^N=106$, $\max(\exc_E^N)=272083$\\
\midrule
47 & [8,24,47,4,0,0]   &  $10^{9}$  &   $\#\exc_E^N=108$, $\max(\exc_E^N)=85963$\\
\midrule
47 & [7,27,55,-2,6,26] &  $10^{9}$  &   $\#\exc_E^N=112$, $\max(\exc_E^N)=78772$\\
\midrule
47 & [4,47,48,0,4,0]   &  $2\times 10^{9}$ & $\#\exc_E^N=556$, $\max(\exc_E^N)=5345827$\\
\midrule
47 & [3,63,63,-2,2,62] &  $10^{9}$  &   $\#\exc_E^N=1165$, $\max(\exc_E^N)=4812283$\\
\midrule
53 & [20,23,32,-12,4,20] & $10^{9}$ & $\#\exc_E^N=30$, $\max(\exc_E^N)=33147$\\
\midrule
53 & [12,19,56,8,12,4]  &  $10^{9}$ &  $\#\exc_E^N=138$, $\max(\exc_E^N)=178027$\\
\midrule
53 & [8,27,55,4,8,2]    &  $10^{9}$ &  $\#\exc_E^N=152$, $\max(\exc_E^N)=137323$\\
\midrule
53 & [3,71,71,-2,2,70]  &  $10^{9}$ &  $\#\exc_E^N=1604$, $\max(\exc_E^N)=6474427$\\
\midrule
59 & [15,16,63,4,2,16] &   $2\times 10^{9}$ & $\#\exc_E^{N}=158$, $\max(\exc_E^N)=304027$\\
\midrule
59 & [15,19,64,-14,8,12] & $2\times 10^{9}$ & $\#\exc_E^{N}=174$, $\max(\exc_E^N)=318091$\\
\midrule
59 & [7,35,68,-6,4,32] &  $2\times 10^{9}$  & $\#\exc_E^{N}=228$, $\max(\exc_E^N)=132883$\\
\midrule
59 & [12,20,59,4,0,0] &   $2\times 10^{9}$  & $\#\exc_E^{N}=193$, $\max(\exc_E^N)=316747$\\
\midrule
59 & [4,59,60,0,4,0] &    $2\times 10^{9}$  & $\#\exc_E^{N}=920$, $\max(\exc_E^N)=3136219$\\
\midrule
59 & [3,79,79,-2,2,78] &  $2\times 10^{9}$  & $\#\exc_E^{N}=2072$, $\max(\exc_E^N)=8447443$\\
\midrule
61 & [23,24,32,16,4,12] & $2\times 10^{9}$ & $\#\exc_E^{N}=43$, $\max(\exc_E^N)=11923$\\
\midrule
61 &  [7,35,71,-2,6,34] & $2\times 10^{9}$ &  $\#\exc_E^{N}=271$, $\max(\exc_E^N)=1096867$\\
\midrule
61 & [8,31,63,4,8,2] &    $2\times 10^{9}$ &  $\#\exc_E^{N}=233$, $\max(\exc_E^N)=363987$\\
\midrule
61 & [11,23,68,-6,8,20] &  $2\times 10^{9}$ & $\#\exc_E^{N}=201$, $\max(\exc_E^N)=190747$\\
\bottomrule
\end{tabular}
\label{exctable3}
\end{table}

\begin{table}
\caption{The set of exceptions $\exc_E^N$.}
\begin{tabular}{llcl}
\toprule
$p$   & Quadratic Form & $N_0/N_1$    & $\exc_E^{N}$ or ($\# \exc_E^{N}$ and $\max(\exc_E^N)$)\\
\midrule
\midrule
67 & [15,36,39,-4,14,16] &  $10^{9}$   & $\#\exc_E^{N}=57$, $\max(\exc_E^N)=20707$\\
\midrule
67 & [23,24,35,8,2,12] &   $10^{9}$    &   $\#\exc_E^{N}=59$, $\max(\exc_E^N)=126043$\\
\midrule
67 & [16,19,71,12,16,6] &    $2\times 10^{9}$ & $\#\exc_E^{N}=264$, $\max(\exc_E^N)=421579$\\
\midrule
67 & [4,67,68,0,4,0] &  $10^{9}$ & $\#\exc_E^{N}=1271$, $\max(\exc_E^N)=3846403$\\
\midrule
71 & [15,20,76,8,4,20] &  $2\times 10^{9}$  & $\#\exc_E^{N}=275$, $\max(\exc_E^N)=321883$\\
\midrule
71 & [15,19,79,-2,14,18] & $2\times 10^{9}$  & $\#\exc_E^{N}=273$, $\max(\exc_E^N)=267883$\\
\midrule
71 & [16,20,71,12,0,0] &   $2\times 10^{9}$  & $\#\exc_E^{N}=310$, $\max(\exc_E^N)=1540771$\\
\midrule
71 & [12,24,71,4,0,0] & $2\times 10^{9}$  & $\#\exc_E^{N}=307$, $\max(\exc_E^N)=635947$\\
\midrule
71 & [8,36,71,4,0,0] &  $2\times 10^{9}$  &   $\#\exc_E^{N}=346$, $\max(\exc_E^N)=1053427$\\
\midrule
71 & [4,71,72,0,4,0] &  $2\times 10^{9}$  &  $\#\exc_E^{N}=1450$, $\max(\exc_E^N)=6463627$\\
\midrule
71 & [3,95,95,-2,2,94] & $2\times 10^{9}$ & $\#\exc_E^{N}=3170$, $\max(\exc_E^N)=15135283$\\
\midrule
73 & [15,39,40,2,8,20] & $10^{9}$ & $\#\exc_E^{N}=81$, $\max(\exc_E^N)=53188$\\
\midrule
73 & [20,31,44,-12,4,28] & $10^{9}$  &  $\#\exc_E^{N}=72$, $\max(\exc_E^N)=111763$\\
\midrule
73 & [7,43,84,-6,4,40] &  $2\times 10^{9}$   & $\#\exc_E^{N}=420$, $\max(\exc_E^N)=364708$\\
\midrule
73 & [11,28,80,8,4,28] &  $2\times 10^{9}$   & $\#\exc_E^{N}=336$, $\max(\exc_E^N)=723795$\\
\midrule
79 & [23,31,44,18,16,20] &  $10^{9}$ & $\#\exc_E^{N}=88$, $\max(\exc_E^N)=50955$\\
\midrule
79 & [16,20,79,4,0,0] &  $2\times 10^{9}$  & $\#\exc_E^{N}=383$, $\max(\exc_E^N)=1419867$\\
\midrule
79 & [19,20,84,16,8,20] & $2\times 10^{9}$ & $\#\exc_E^{N}=391$, $\max(\exc_E^N)=1210675$\\
\midrule
79 & [11,31,87,-10,6,26] & $2\times 10^{9}$ & $\#\exc_E^{N}=409$, $\max(\exc_E^N)=12778803$\\
\midrule
79 & [8,40,79,4,0,0] & $2\times 10^{9}$    & $\#\exc_E^{N}=495$, $\max(\exc_E^N)=1116507$\\
\midrule
79 & [4,79,80,0,4,0] & $2\times 10^{9}$    & $\#\exc_E^{N}=1886$, $\max(\exc_E^N)=25575460$\\
\midrule
83 & [23,31,44,-14,8,12] & $10^{9}$    & $\#\exc_E^{N}=97$, $\max(\exc_E^N)=36763$\\
\midrule
83 & [12,28,83,4,0,0] & $2\times 10^{9}$    & $\#\exc_E^{N}=432$, $\max(\exc_E^N)=635347$\\
\midrule
83 & [7,48,95,4,2,48] & $2\times 10^{9}$    & $\#\exc_E^{N}=529$, $\max(\exc_E^N)=1358107$\\
\midrule
83 & [16,23,87,12,16,6] & $2\times 10^{9}$  & $\#\exc_E^{N}=416$, $\max(\exc_E^N)=1202587$\\
\midrule
83 & [11,31,92,-6,8,28] & $2\times 10^{9}$ & $\#\exc_E^{N}=469$, $\max(\exc_E^N)=1381867$\\
\bottomrule
\end{tabular}
\label{exctable4}
\end{table}

\begin{table}
\caption{The set of exceptions $\exc_E^N$.}
\begin{tabular}{llcl}
\toprule
$p$   & Quadratic Form & $N_0/N_1$    & $\exc_E^{N}$ or ($\# \exc_E^{N}$ and $\max(\exc_E^N)$)\\
\hline
\midrule
83 & [3,111,111,-2,2,110] & $2\times 10^{9}$ & $\#\exc_E^{N}=4639$, $\max(\exc_E^N)=62337067$\\
\midrule
83 & [4,83,84,0,4,0] &  $2\times 10^{9}$ & $\#\exc_E^{N}=2134$, $\max(\exc_E^N)=9405643$\\
\midrule
89 & [23,31,48,2,12,16] &  $10^{9}$ &    $\#\exc_E^{N}=118$, $\max(\exc_E^N)=137707$\\
\midrule
89 & [15,24,95,4,2,24] &     $5\times 10^8$ & $\#\exc_E^{N}=502$, $\max(\exc_E^N)=682147$\\
\midrule
89 & [15,27,96,-14,8,20] &   $5\times 10^8$ & $\#\exc_E^{N}=464$, $\max(\exc_E^N)=1534723$\\
\midrule
89 & [19,23,95,-18,10,14] &  $5\times 10^8$ & $\#\exc_E^{N}=540$, $\max(\exc_E^N)=981403$\\
\midrule
89 & [7,51,103,-2,6,50] &    $5\times 10^8$ & $\#\exc_E^{N}=646$, $\max(\exc_E^N)=1427827$\\
\midrule
89 & [3,119,119,-2,2,118] &  $2\times 10^9$ & $\#\exc_E^{N}=5357$, $\max(\exc_E^N)=28654707$\\
\midrule
89 & [12,31,92,8,12,4] &    $5\times 10^8$  & $\#\exc_E^{N}=478$, $\max(\exc_E^N)=653227$\\
\midrule
97 & [23,39,51,-22,6,14] &    $10^9$  & $\#\exc_E^{N}=283$, $\max(\exc_E^N)=74011$\\
\midrule
97 & [15,52,55,-4,14,24] &    $10^9$  & $\#\exc_E^{N}=295$, $\max(\exc_E^N)=94963$\\
\midrule
97 & [7,56,111,4,2,56] &    $10^9$  & $\#\exc_E^{N}=814$, $\max(\exc_E^N)=851272$\\
\midrule
97 & [20,39,59,-4,8,38] &    $10^9$  & $\#\exc_E^{N}=277$, $\max(\exc_E^N)=118243$\\
\midrule
97 & [19,23,104,-14,12,16] &    $10^9$  & $\#\exc_E^{N}=636$, $\max(\exc_E^N)=1336483$\\
\midrule
101 & [32,39,44,-12,28,20] & $10^9$ & $\#\exc_E^{N}=158$, $\max(\exc_E^N)=123523$\\
\midrule
101 & [12,35,104,8,12,4] & $10^9$ & $\#\exc_E^{N}=652$, $\max(\exc_E^N)=1157083$\\
\midrule
101 & [15,28,108,8,4,28] & $10^9$ & $\#\exc_E^{N}=625$, $\max(\exc_E^N)=1299163$\\
\midrule
101 & [15,27,111,-2,14,26] & $10^9$ & $\#\exc_E^{N}=652$, $\max(\exc_E^N)=901363$\\
\midrule
101 & [7,59,116,-6,4,56] & $10^9$ & $\#\exc_E^{N}=881$, $\max(\exc_E^N)=1720048$\\
\midrule
101 & [11,39,111,-10,6,34] & $10^9$ & $\#\exc_E^{N}=723$, $\max(\exc_E^N)=1305627$\\
\midrule
101 & [3,135,135,-2,2,134] & $10^9$ & $\#\exc_E^{N}=7304$, $\max(\exc_E^N)=24487147$\\
\midrule
103 & [23,36,59,-4,22,16] & $10^9$ & $\#\exc_E^{N}=174$, $\max(\exc_E^N)=121027$\\
\midrule
103 & [16,28,103,-12,0,0] & $10^9$ & $\#\exc_E^{N}=696$, $\max(\exc_E^N)=1004347$\\
\midrule
103 & [15,55,56,2,8,28] & $10^9$ & $\#\exc_E^{N}=200$, $\max(\exc_E^N)=353728$\\
\bottomrule
\end{tabular}
\label{exctable5}
\end{table}

\begin{table}
\caption{The set of exceptions $\exc_E^N$.}
\begin{tabular}{llcl}
\toprule
$p$   & Quadratic Form & $N_0/N_1$    & $\exc_E^{N}$ or ($\# \exc_E^{N}$ and $\max(\exc_E^N)$)\\
\hline
\midrule
103 & [19,23,111,-10,14,18] & $10^9$ & $\#\exc_E^{N}=709$, $\max(\exc_E^N)=1086547$\\
\midrule
103 & [8,52,103,4,0,0] & $10^9$ & $\#\exc_E^{N}=896$, $\max(\exc_E^N)=1019467$\\
\midrule
103 & [7,59,119,-2,6,58] & $10^9$ & $\#\exc_E^{N}=903$, $\max(\exc_E^N)=1959163$\\
\midrule
103 & [4,103,104,0,4,0] & $10^9$ & $\#\exc_E^{N}=2358$, $\max(\exc_E^N)=6390532$\\
\midrule
107 &   [35,39,44,-18,32,4] & $10^9$ & $\#\exc_E^{N}=186$, $\max(\exc_E^N)=169467$\\
\midrule
107 &   [23,40,56,16,4,20] & $10^9$ & $\#\exc_E^{N}=209$, $\max(\exc_E^N)=274387$\\
\midrule
107 &   [16,27,111,-4,16,2] & $10^9$ & $\#\exc_E^{N}=769$, $\max(\exc_E^N)=2998675$\\
\midrule
107 &   [12,36,107,4,0,0] & $10^9$ & $\#\exc_E^{N}=817$, $\max(\exc_E^N)=695179$\\
\midrule
107  &   [19,23,116,-6,16,20] & $10^9$ & $\#\exc_E^{N}=813$, $\max(\exc_E^N)=3142483$\\
\midrule
107 &   [11,39,119,-2,10,38] & $10^9$ & $\#\exc_E^{N}=856$, $\max(\exc_E^N)=838987$\\
\midrule
107 &   [4,107,108,0,4,0] & $10^9$ & $\#\exc_E^{N}=3873$, $\max(\exc_E^N)=13204228$\\
\midrule
107 &   [3,143,143,-2,2,142] & $10^9$ & $\#\exc_E^{N}=8410$, $\max(\exc_E^N)=44363163$\\
\midrule
109 & [32,44,47,20,28,36] & $10^8$ & $\#\exc_E^{N}=205$, $\max(\exc_E^N)=193747$\\ 
\midrule
109 & [23,39,59,10,14,22] & $10^8$ & $\#\exc_E^{N}=215$, $\max(\exc_E^N)=1034083$\\ 
\midrule
109 & [8,55,111,4,8,2] & $10^9$ & $\#\exc_E^{N}=1039$, $\max(\exc_E^N)=2522587$\\ 
\midrule
109 & [24,39,56,16,12,4] &  $10^8$ & $\#\exc_E^{N}=225$, $\max(\exc_E^N)=215659$\\ 
\midrule
109 & [11,40,119,4,2,40] &  $10^9$ & $\#\exc_E^{N}=891$, $\max(\exc_E^N)=947755$\\
\midrule
109 & [19,23,119,-2,18,22] &  $10^9$ & $\#\exc_E^{N}=857$, $\max(\exc_E^N)=1300915$\\
\midrule
113 & [35,39,47,-6,34,10] & $10^8$ & $\#\exc_E^{N}=213$, $\max(\exc_E^N)=142267$\\
\midrule
113 & [23,40,59,8,2,20] & $10^8$ & $\#\exc_E^{N}=220$, $\max(\exc_E^N)=146787$\\
\midrule
113 & [20,47,68,-12,4,44] & $10^8$ & $\#\exc_E^{N}=247$, $\max(\exc_E^N)=253363$\\
\midrule
113 & [23,24,119,20,10,24] & $5\times 10^9$ & $\#\exc_E^{N}=904$, $\max(\exc_E^N)=1800643$\\
\midrule
113 & [19,24,119,4,2,24] & $5\times 10^9$ & $\#\exc_E^{N}=1005$, $\max(\exc_E^N)=1997835$\\
\midrule
113 & [12,39,116,8,12,4] & $5\times 10^9$ & $\#\exc_E^{N}=907$, $\max(\exc_E^N)=1130803$\\
\midrule
113 & [3,151,151,-2,2,150] & $5\times 10^9$ & $\#\exc_E^{N}=9302$, $\max(\exc_E^N)=30158683$\\
\bottomrule
\end{tabular}
\label{exctable6}
\end{table}
 
\begin{table}
\caption{Good Bounds $D_p$ from Theorem \ref{BoundTheorem} for $p\leq 107$.}
\begin{tabular}{p{2cm}cp{2.5cm}p{2cm}c}
    \toprule
$p$ & $D_p$& &$p$ & $D_p$\\
\cmidrule{1-2}\cmidrule{4-5}
$3,5,7,13$ & $1$ & & $59$ & $1.166\times 10^{19}$\\
$11$ & $8.973\times 10^{9}$ & & $67$ & $2.642\times 10^{19}$\\
$17$ & $1.510\times 10^{15}$ & &$71$ & $1.793\times 10^{21}$\\
$19$ & $1.606\times 10^{13}$ & &$73$ & $1.452\times 10^{19}$\\
$23$ & $6.955\times 10^{15}$ & &$79$ & $2.370\times 10^{20}$\\
$29$ & $1.550\times 10^{17}$ & &$83$ & $1.033\times 10^{20}$\\
$31$ & $1.008\times 10^{16}$ & &$89$ & $3.257\times 10^{25}$\\
$37$ & $1.117\times 10^{15}$ & &$97$ & $7.815\times 10^{19}$\\
$41$ & $2.379\times 10^{17}$ & &$101$ & $5.296\times 10^{20}$\\
$43$ & $6.461\times 10^{18}$ & &$103$ & $8.748\times 10^{19}$\\
$47$ & $5.467\times 10^{19}$ & &$107$ & $1.761\times 10^{21}$\\
$53$ & $1.166\times 10^{19}$ & &$109$ & $1.841\times 10^{21}$\\
$61$ & $3.797\times 10^{17}$ & &$113$ & $1.041\times 10^{23}$\\
\cmidrule{1-2}\cmidrule{4-5}
\end{tabular}
\label{GoodBoundTable}
\end{table}

\bibliographystyle{amsplain}
\bibliography{biblio}

\providecommand{\bysame}{\leavevmode\hbox to3em{\hrulefill}\thinspace}
\providecommand{\MR}{\relax\ifhmode\unskip\space\fi MR }
\providecommand{\MRhref}[2]{%
  \href{http://www.ams.org/mathscinet-getitem?mr=#1}{#2}
}
\providecommand{\href}[2]{#2}
\begin{thebibliography}{10}

\bibitem{Cremona1}
J.~Cremona, \emph{Algorithms for elliptic curves}, Cambridge Univ. Press, 1992.

\bibitem{Deligne1}
P.~Deligne, \emph{La conjecture de weil i}, Inst. Hautes \'Etudes Sci. Publ.
  Math \textbf{43} (1974), 273--307.

\bibitem{Deuring1}
M.~Deuring, \emph{Die {T}ypen der {M}ultiplikatorenringe elliptischer
  {F}unktionenk\"orpen}, Abh. Math. Sem. Hansischen Univ. \textbf{14} (1941),
  197--272.

\bibitem{Duke1}
W.~Duke, \emph{Hyperbolic distribution problems and half-integral weight maass
  forms}, Invent. Math. \textbf{92} (1998), 73--90.

\bibitem{DukePillot1}
W.~Duke and R.~Schulze-Pillot, \emph{Representation of integers by positive
  ternary quadratic forms and equidistribution of lattice points on
  ellipsoids}, Invent. Math. \textbf{99} (1990), no.~1, 49--57.

\bibitem{Elkies1}
N.~Elkies, \emph{Supersingular primes for elliptic curves over real number
  fields}, Compositio Mathematica \textbf{72} (1989), 165--172.

\bibitem{ElkiesOnoYang1}
N.~Elkies, K.~Ono, and T.~Yang, \emph{Reduction of {CM} elliptic curves and
  modular function congruences}, Int. Math. Res. Not. \textbf{44} (2005),
  2695--2707.

\bibitem{FinckePohst1}
U.~Fincke and M.~Pohst, \emph{Improved methods for calculating vectors of short
  length in a lattice, including a complexity analysis}, Math. Comp. (1985),
  463--471.

\bibitem{incomplete}
W.~Gautschi, \emph{A computational procedure for incomplete {G}amma functions},
  ACM Transactions on Mathematical Software \textbf{5} (1979), 466--481.

\bibitem{Gross1}
B.~Gross, \emph{Heights and the special values of {$L$}-series}, Number theory
  (Montreal, Que., 1985), CMS Conf. Proc., vol.~7, Amer. Math. Soc.,
  Providence, RI, 1987, pp.~115--187.

\bibitem{GrossZagier1}
B.~Gross and D.~Zagier, \emph{On singular moduli}, J. Reine Angew. Math.
  \textbf{335} (1985), 191--220.

\bibitem{Ibukiyama1}
T.~Ibukiyama, \emph{On maximal order of division quaternion algebras over the
  rational number field with certain optimal embeddings}, Nagoya Math J.
  \textbf{88} (1982), 181--195.

\bibitem{Iwaniec1}
H.~Iwaniec, \emph{Fourier coefficients of modular forms of half-integral
  weight}, Invent. Math. \textbf{87} (1987), 385--401.

\bibitem{Jones1}
B.~Jones, \emph{The arithmetic theory of quadratic forms}, Carcus Monograph
  Series, no.~10, The Mathematical Association of America, Buffalo, Buffalo,
  NY, 1950.

\bibitem{Kane2}
B.~Kane, \emph{Representations of integers by ternary quadratic forms},
  preprint (2007).

\bibitem{Kohel1}
D.~Kohel, \emph{Endomorphism rings of elliptic curves over finite fields},
  University of California, Berkeley, Ph.D. Thesis (1996), pp. 1--96.

\bibitem{Kohnen1}
W.~Kohnen, \emph{Newforms of half integral weight}, J. reine angew. Math.
  \textbf{333} (1982), 32--72.

\bibitem{KohnenZagier1}
W.~Kohnen and D.~Zagier, \emph{Values of {$L$}-series of modular forms at the
  center of the critical strip}, Invent. Math. \textbf{64} (1981), 175--198.

\bibitem{Oesterle1}
J.~Oesterl\'{e}, \emph{Nombres de classes des corps quadratiques imaginaires},
  Ast\'{e}rique \textbf{121-122} (1985), 309--323.

\bibitem{o'meara}
O.~T. O'Meara, \emph{Introduction to quadratic forms}, Classics in Mathematics,
  Springer-Verlag, Berlin, 2000, Reprint of the 1973 edition.

\bibitem{Ono1}
K.~Ono, \emph{Web of modularity: Arithmetic of the coefficients of modular
  forms and $q$-series}, CBMS Regional Conference Series in Mathematics, no.
  102, Amer. Math. Soc., Providence, RI, 2003.

\bibitem{OnoSound1}
K.~Ono and K.~Soundarajan, \emph{Ramanujan's ternary quadratic form}, Invent.
  Math. \textbf{130} (1997), 415--454.

\bibitem{Shimura1}
G.~Shimura, \emph{On modular forms of half integer weight}, Ann. of Math.
  \textbf{97} (1973), 440--481.

\bibitem{Siegel1}
C.~Siegel, \emph{\"{U}ber die klassenzahl quadratischer zahlkorper}, Acta
  Arith. \textbf{1} (1935), 83--86.

\bibitem{Silverman1}
J.~Silverman, \emph{The arithmetic of elliptic curves}, Springer-Verlag, New
  York, 1992, Corrected reprint of the 1986 original.

\bibitem{Stein1}
W.~Stein, \emph{Explicit approaches to modular abelian varieties}, Ph.D.
  thesis, University of California, Berkeley (2000), pp. 1--96.

\bibitem{Stein2}
\bysame, \emph{Modular forms, a computational approach}, Graduate Studies in
  Mathematics, vol.~79, American Mathematical Society, Providence, RI, 2007,
  Appendix by P. Gunnells.

\bibitem{Sturm1}
J.~Sturm, \emph{On the congruence of modular forms}, Number theory (New York,
  1984--1985), Springer, Berlin, 1987, pp.~275--280.

\bibitem{Vigneras1}
M.-F. Vign{\'e}ras, \emph{Arithm\'etique des alg\`ebres de quaternions},
  Lecture Notes in Mathematics, vol. 800, Springer, Berlin, 1980.

\end{thebibliography}

\end{document}